\documentclass[11pt]{amsart}
\usepackage{pdfsync}
\usepackage{epsfig}
\usepackage{amsmath}
\usepackage{amssymb}
\usepackage{euscript}
\usepackage{enumerate}
\usepackage{pxfonts}
\usepackage{color}
\usepackage{graphics}

\newcommand{\R}{{\mathbb R}}

\newcommand{\dist}{{\text{\rm dist}}}

\newcommand{\e}{\epsilon}

\newcommand{\cA}{{\mathcal A}}

\newcommand{\cM}{{\mathcal M}}

\newcommand{\cS}{{\mathcal S}}

\newcommand{\osc }{\text{\rm osc}}

\newtheorem{theorem}[subsubsection]{Theorem}
\newtheorem{lemma}[subsubsection]{Lemma}
\newtheorem{cor}[subsubsection]{Corollary}
\newtheorem{remark}[subsubsection]{Remark}
\newtheorem{prop}[subsubsection]{Proposition}
\newtheorem{example}[subsubsection]{Example}
\newtheorem{definition}[subsubsection]{Definition}

\numberwithin{equation}{section}

\DeclareMathOperator{\tr}{tr}

\title[Homogenization with Oscillating Boundary Data]{Homogenization of Fully Nonlinear Elliptic Equations with Oscillating Dirichlet Boundary Data}
\begin{document}

\author{Ki-ahm Lee}
\address{Seoul National University, Seoul 151-747, Korea}
\email{kiahm@snu.ac.kr}
\author{Minha Yoo}
\address{Seoul National University, Seoul 151-747, Korea}
\email{minha00@snu.ac.kr}

\maketitle
\begin{abstract}
This paper deals with the homogenization of fully nonlinear second order equation with an oscillating Dirichlet boundary data when the operator and boundary data are $\e$-periodic. We will show that the solution $u_\e$ converges to some function $\overline u(x)$ uniformly on every compact subset $K$ of the domain $D$. Moreover, $\overline u$ is a solution to some boundary value problem. For this result, we assume that the boundary of the domain has no (rational) flat spots and the ratio of elliptic constants $\Lambda / \lambda$ is sufficiently large. 
\end{abstract}

\section{Introduction}\label{sec-intro}
In this paper, we are going to consider the homogenization problem with oscillating Dirichlet boundary data. Let $D$ be a bounded domain in $\R^n$ and let $u_\e$ be the viscosity solution of the following equation,
\begin{equation} \label{eq-main} \tag{$P_{\e}$}
\begin{cases}
F\left(D^2u_\e(x),\frac{x}{\e}  \right)= f\left( x,\frac{x}{\e} \right) &\text{ in } D \\
u_\e(x)=g\left(x, \frac{x}{\e} \right)&\text{ on } \partial  D.
\end{cases}
\end{equation}
Here $f(x,y)$ is a continuous and uniformly bounded function, $g$ is a $\mathcal C^2$-function in $(x,y)$ which is $\mathcal C^{2,\alpha}$ for fixed $x\in D$ satisfying 
\begin{equation} \label{con-g}
\sup_{x \in \partial D} \| g(x,\cdot) \|_{\mathcal{C}^{2,\alpha}(\R^n)} < \infty,
\end{equation}
and $F$ is a continuous function satisfying the following two conditions:
\begin{enumerate}
\item (Uniformly Ellipticity) 
For given any symmetric matrix $M$ and positive symmetric matrix $N$, there are constants $0<\lambda\le\Lambda < +\infty$ satisfying
\begin{equation} \label{con-F-unif}
\lambda ||N|| \leq F(M+N,y)-F(M,y) \leq \Lambda ||N|| \qquad\\
\end{equation}
where $\| \cdot \|$ is a matrix norm defined by $\|M\| = \sum_{1\le i,j\le n} M_{ij}^2$.
\item (Positive Homogenity)
For given any $t>0$ and any symmetric matrix $M$, 
\begin{equation} \label{con-phomo}
F(t M,y)  = t F(M,y).
\end{equation} 
\end{enumerate}

We assume that $F(M,y)$, $f(x,y)$ and $g(x,y)$ are periodic in the $y$-variable, that is, $F(M,y+k) = F(M,y)$, $f(x,y+k)=f(x,y)$, and $g(x,y+k)=g(x,y)$ for all $M \in \mathcal{S}$, $x \in \overline D$, $y\in \R^n$ and $k \in \mathbb{Z}^n$ where $\mathcal{S}$ denotes the set of all $n\times n$ symmetric matrices.

According to \cite{E}, if there exists a limit $\overline u$ of $u_\e$, then there are homogenized functions $\overline F$ and $\overline f$ which are independent on the $y$ variable and the limit $\overline u$ satisfies
\begin{equation} \label{eq-lim}
\overline F(D^2 \overline u) = \overline f(x) \text{ in } D
\end{equation}
in the viscosity sense. Moreover, if the boundary data $g(x,x/\e)$ does not depend on the $y$ variable, then the solution $u_\e$ converges to the solution $\overline u$ of the equation \eqref{eq-lim} equipped the same boundary data $g(x)$ uniformly on the $x$ variable. However, because of the oscillation in the boundary data, the uniform convergence on $\overline D$ cannot be expected in our case and hence much more delicate analysis is needed. 

We say that a vector $\nu \in S^{n-1} =\{ \nu \in \R^n : |\nu|=1 \}$ is rational if $t \nu \in \mathbb Z^n$ for some $t \in \R^n$ and that a vector $\nu$ is irrational if it is not rational. Additionally, we call a point $x \in \partial D$ a rational point if its outward unit normal vector $\nu(x)$ is rational. We define a irrational point in the similar way. Finally, we say a domain $D$ satisfies the Irrational Direction Dense Condition, IDDC, if all but countably many points on $\partial D$ are irrational. The simplest domains satisfying the IDDC is a ball $B_r(0)$, $r>0$. The formal definition for IDDC can be found in the Section \ref{sec-cor}.

To exist the  limit of $u_\e$ uniquely, IDDC is necessary. We will give an example in Section \ref{sec-thm1} that fails the uniqueness of the limit if the domain $D$ does not hold the IDDC. Further information on IDDC can be found in \cite{LS} and {\color{blue}\cite{GM}.}

We are going to define the effective boundary data $\overline g(x)$ on $\partial D$ in Section \ref{sec-hyper} and Section \ref{sec-cor}. Unfortunately, $\overline g$ is defined only when $x \in \partial D$ is a irrational point. As long as we know, there are no concepts of the viscosity solutions with discontinuous boundary data. So, we need to define the following,
\begin{definition} \label{def-gsol}
Let $g$ be a function defined on $\partial D$ except countably many points, $g^\pm$ be continuous functions defined on $\partial D$ and $u^\pm$ be viscosity solutions of
\begin{equation} \label{eq-def-gsol} \begin{cases}
F(D^2 v(x))=f(x) &\text{ in } D \\
v(x) = g(x) &\text{ on } \partial D\\
\end{cases} \end{equation}
when the  boundary condition $g$ is replaced by $g^\pm$ respectively where $g^\pm$ are continuous functions defined on $\partial D$. 
We say $v$ is a (viscosity) solution of the equation \eqref{eq-def-gsol} in the general sense if $v$ satisfies $u^-(x) \le v(x) \le u^+(x)$ in $D$ for any $g^\pm(x)$ satisfying $g^-(x) \le g(x) \le g^+(x)$ on $\{x \in \partial D : g(x) \text{ is defined}\}$.
\end{definition}

\begin{theorem}\label{thm-main-2} 
Suppose that $u_\e$ is a solution of the equation \eqref{eq-main}. Additionally, suppose that 
\begin{enumerate}
\item
the domain $D$ satisfies the IDDC,
\item 
the equation \eqref{eq-main} is  in the stable class for the finite values of the boundary. 
\end{enumerate}
Then $\overline g_*(x)=\overline g^*(x)$ for all the irrational points and there is a function $\overline u$ such that $u_\e$, the solution of \eqref{eq-main}, converges to $\overline u$ uniformly on every compact set $K \subset D$ where $\overline g^*$ and $\overline g_*$ are same in the Definition \ref{def-cor-overg}.

Moreover, $\overline u$ is a unique solution of the following equation  
\begin{equation} \label{eq-limit-2}
\begin{cases}
\overline{F}(D^2 \overline u(x))=\overline f(x) &\text{ in } D \\
\overline u(x) = \overline g(x) &\text{ on } \partial D\\
\end{cases}
\end{equation}
in the general sense where $\overline g = \overline g^*$.
\end{theorem}

You can find the definition that \eqref{eq-main} is  in the stable class for finite values of the boundary in Section \ref{sec-overg}. See the Definition \ref{def-overg-finite}. Heuristically, it implies that even if we change the value of $g$ at some finite points, the solution does not change. It is obvious in the Laplace equation because the solution can be represented as a integral on the boundary and the finite values of the boundary data $\overline g$ at finite points of the boundary is measure zero. However, it is not obvious in the Fully nonlinear elliptic equations. We only find the sufficient condition in Section \ref{sec-overg} but it is open whether the general Fully nonlinear equations are contained in the stable class for finite values of the boundary. 

\begin{remark}
Our argument can be applied if $F=F(M,y)$ depends on the $x$ variable. However, for simplicity, we only consider the case when $F$ is independent on the $x$ variable. 
\end{remark} 

In Section \ref{sec-vis}, we summarize the existence and regularity theory of the viscosity solution. In Section \ref{sec-hyper}, we define a corrector, a function on the half plain, and investigate their properties. By using the corrector, we define the effective boundary data $\overline g$ and we measure how the solution $u_\e$ and the corrector $w_\e$ are close in Section \ref{sec-cor}. In section \ref{sec-gconti}, the continuity of $\overline g$ will be discussed, and then we will focus on the proof of the Theorem \ref{thm-main-2} in the remaining section. 

The existence and uniqueness of a viscosity solution can be found in \cite{CIL} and its regularity theory can be found in \cite{CC}, \cite{GT} and \cite{LT}. The interior homogenization result for fully nonlinear equation can be found \cite{E} in a periodic case, \cite{CSW}, \cite{LS1}, and \cite{LS2} in a random one. We also refer \cite{JKO} for the linear homogenization. The oscillating boundary data for the divergence equation can be considered in \cite{GM} and \cite{AL} and the homogenization of oscillating Neumann boundary data can be found in \cite{CKL} and \cite{BDLS}. In \cite{LS}, the authors showed the similar theorem for Laplace operator, or operators of divergence type with Green representation. In this paper, we try to show similar result for nonlinear non-divergence operator which require very different approach due to the lack of the representation. 

\section{General Facts of the Viscosity Solution}\label{sec-vis}

We say a continuos function $u \in \mathcal{C}^0(\overline D)$ is a viscosity super-solution of the equation 
\begin{equation} \label{eq-vis-1}
F(D^2u,Du,u,x) = f(x) 
\end{equation}
in $D$ if there exists a function $\varphi(x) \in \mathcal{C}^2$ which is defined some neighborhood of $x_0 \in D$ and $u-\varphi$ has a local maximum at $x_0 \in D$, then 
\begin{equation}
F(D^2 \varphi(x_0),D\varphi(x_0),u(x_0),x_0) \le f(x_0).
\end{equation}
We define the viscosity sub-solution in the similar way and we say $u$ is a viscosity solution of \eqref{eq-vis-1} if $u$ is a viscosity sub and super solution. 

The existence of viscosity solution is given at \cite{CIL}. 
\begin{theorem}[Existence and Uniqueness, \cite{CIL}] \label{thm-vis-ex}
There exists a unique continuous viscosity solution of the following equation,
\begin{equation} \begin{cases}\label{eq-vis-2}
F(D^2u,Du,u,x) = f(x) &\text{ in } D \\ 
u(x)=g(x) &\text{ on } \partial D 
\end{cases} \end{equation}
for any given continuous and bounded function $f(x)$ and $g(x)$ if the operator $F$ satisfies the structure condition in \cite{CIL}.
\end{theorem}
The structure conditions for $F$ can be found at \cite{CIL}. Becauses of the condition \ref{con-F-unif} and the continuity of the operator, we can find a viscosity solution of the equation \eqref{eq-main} for each $\e > 0$. 

\begin{lemma}[Comparison, \cite{CIL}] \label{vis-com}
Suppose that $u$ is a viscosity super-solution of the equation \eqref{eq-vis-2} and $v$ is a viscosity sub-solution of the same equation. Suppose also that $u \ge v$ on $\partial D$. Then we have $u \ge v$ in $\overline D$. 
\end{lemma}

The boundedness of $D$ is not necessary since Theorem \ref{thm-vis-ex} and Lemma \ref{vis-com} holds even for unbounded domains. For example, there is a viscosity solution when the domain is a half-plain. We refer \cite{CLV} for details.  

The following results in \cite{CC} will be used frequently in this paper.
\begin{prop}[\cite{CC}] \label{prop-vis-osc-holder}
Suppose that $u$ is a viscosity solution of \eqref{eq-vis-1} in $B_1(0)$. Then,
\begin{enumerate}
\item
there exists a constant $0<\gamma<1$, depending only on the dimension and the elliptic constants such that 
\begin{equation} \label{eq-vis-osc}
\osc_{B_{1/2}(0)} u \le \gamma \osc_{B_1(0)} u + \| f\|_{L^n(B_1(0))},
\end{equation}
\item
and then $u$ is in $\mathcal{C}^\alpha(\overline B_{1/2})$ with
\begin{equation} \label{eq-vis-holder}
\| u \|_{\mathcal{C}^\alpha(\overline B_{1/2}(0))} \le C\left( \| u\|_{L^\infty(B_1(0))} + \|f\|_{L^n(B_1)}\right)
\end{equation}
where $0 < \alpha <1$ and a constant $C$ depending only on the dimension and the elliptic constants $\lambda$ and $\Lambda$. 
\end{enumerate}
\end{prop}

We note that the domain $B_1(0)$ in the Proposition above can be changed to general domain $D$ and $B_{1/2}(0)$ also can be replaced by $K$ such $\overline K \subset D$ by using the covering argument. In this case, the constant $C$ depends on $K$ and $D$. 

For $M\in\cS^n$ and $0<\lambda \leq \Lambda,$  the Pucci's extremal operators, playing a crucial role in the study of fully nonlinear elliptic equations, are defined as
\begin{equation} \label{def-vis-pucci} \begin{aligned}
\cM^+_{\lambda,\Lambda}(M)&=\cM^+(M)=\sup_{A\in\cA_{\lambda,\Lambda}} [ \tr(AM)] \\ \cM^-_{\lambda,\Lambda}(M)&=\cM^-(M)=\inf_{A\in\cA_{\lambda,\Lambda}}[\tr(AM)]
\end{aligned} \end{equation}
where $\cA_{\lambda,\Lambda}$ consists of the symmetric matrices, the eigenvalues of which lie in $[\lambda,\Lambda]$. Note that for $\lambda=\Lambda=1,$  the Pucci`s extremal operators $\cM^\pm$ simply coincide with the Laplace operator.

Let $\mathcal{S}(\lambda, \Lambda)$ be the family of all functions $u$ satisfying 
\begin{equation}
\cM^+(D^2u) \ge 0,\text{ and } \cM^-(D^2 u) \le 0
\end{equation}
in the viscosity sense. We note that all the viscosity solutions of \eqref{eq-vis-1} for $f=0$ are in $\mathcal{S}(\lambda,\Lambda)$ and all the functions in $\mathcal{S}(\lambda,\Lambda)$ satisfy the result in the Proposition \ref{prop-vis-osc-holder}. So, roughly speaking, if $u$ is in $\mathcal{S}(\lambda,\Lambda)$, then $u$ is a viscosity solution of some uniformly elliptic operator in the same class. 
\begin{theorem}[\cite{CC}] \label{thm-vis-diff}
Suppose that $u$ and $v$ are viscosity solutions of \eqref{eq-vis-1}. Suppose also that $F$ in \eqref{eq-vis-1} is independent of $Du$ and $u$ variables. Then, 
\begin{equation}
u-v \in \mathcal{S}(\lambda/n,\Lambda). 
\end{equation}
Hence, the maximum principle and the results in Proposition \ref{prop-vis-osc-holder} also valid for $u-v$ as a viscosity solution of some uniformly elliptic equation.
\end{theorem}

The proof of Theorem above can be found in chapter 5 of \cite{CC}. Although, in \cite{CC}, they considered the case when \eqref{eq-vis-1} is independent of the $x$ variable, the same proof also holds in our case.

\section{Functions Defined on a Half-plain}\label{sec-hyper}
In this section, we define a corrector to describe the effective boundary data. 
\begin{definition} \item \label{def-hyper-irr}
\begin{enumerate}
\item
A vector $\nu \in S^{n-1}$ is irrational if $\frac{\nu_i}{\nu_j}$ is an irrational number for some $1\le i,j \le n $. 
\item
A vector $\nu \in S^{n-1}$ is rational if it is not irrational. 
\end{enumerate}
\end{definition}
The above definition is equivalent to the definition of rational and irrational direction in the introduction. We denote $\mathcal R$ as the set of all rational directions in $S^{n-1}$ and $\mathcal{IR}$ as the set of all irrational directions in $S^{n-1}$. Note that the number of elements of $\mathcal R$ is countable. 

According to \cite{LSY}, every irrational vector $\nu$ has a averaging property. The following Lemma is a modification of the Lemma 5.2.2 in \cite{LSY} for the uniform distribution.
\begin{lemma}[\cite{LSY}] \label{hyper-ud}
Let $Q^\prime_R$ be any cube with side length $R>0$ in $\R^{n-1}$. Suppose that $h$ is a function defined on $Q^\prime_R$ such that 
\begin{equation}
h(y^\prime) = \alpha_1 y_1 + \alpha_2 y_2 + \cdots + \alpha_{n-1} y_{n-1} 
\end{equation}
for some $(\alpha_1,\alpha_2, \cdots,\alpha_{n-1}) \in \R^{n-1}$. Denote that 
\begin{equation} \begin{aligned}
N( R) &= \#\left( Q_R^\prime \cap \mathbb{Z}^{n-1} \right) \text{ and }\\
A(\delta,t,R) &= \# \{ m \in Q_R^\prime \cap \mathbb{Z}^{n-1} : h(m)/\mathbb{Z} \in [t,t+\delta)/\mathbb{Z} \}.
\end{aligned} \end{equation}
If one of $\alpha_i$ ( $i=1,2,\cdots,n-1$ ) is irrational, then there exists a modulus of continuity $\rho$ such that $\rho(0+) = 0$ and 
\begin{equation}
\left| \displaystyle\frac{A(\delta,t,R)}{N( R)} -\delta \right|\le\rho\left(\frac{1}{R}\right).
\end{equation}
\end{lemma}

\begin{figure}
\includegraphics[width=0.75\textwidth]{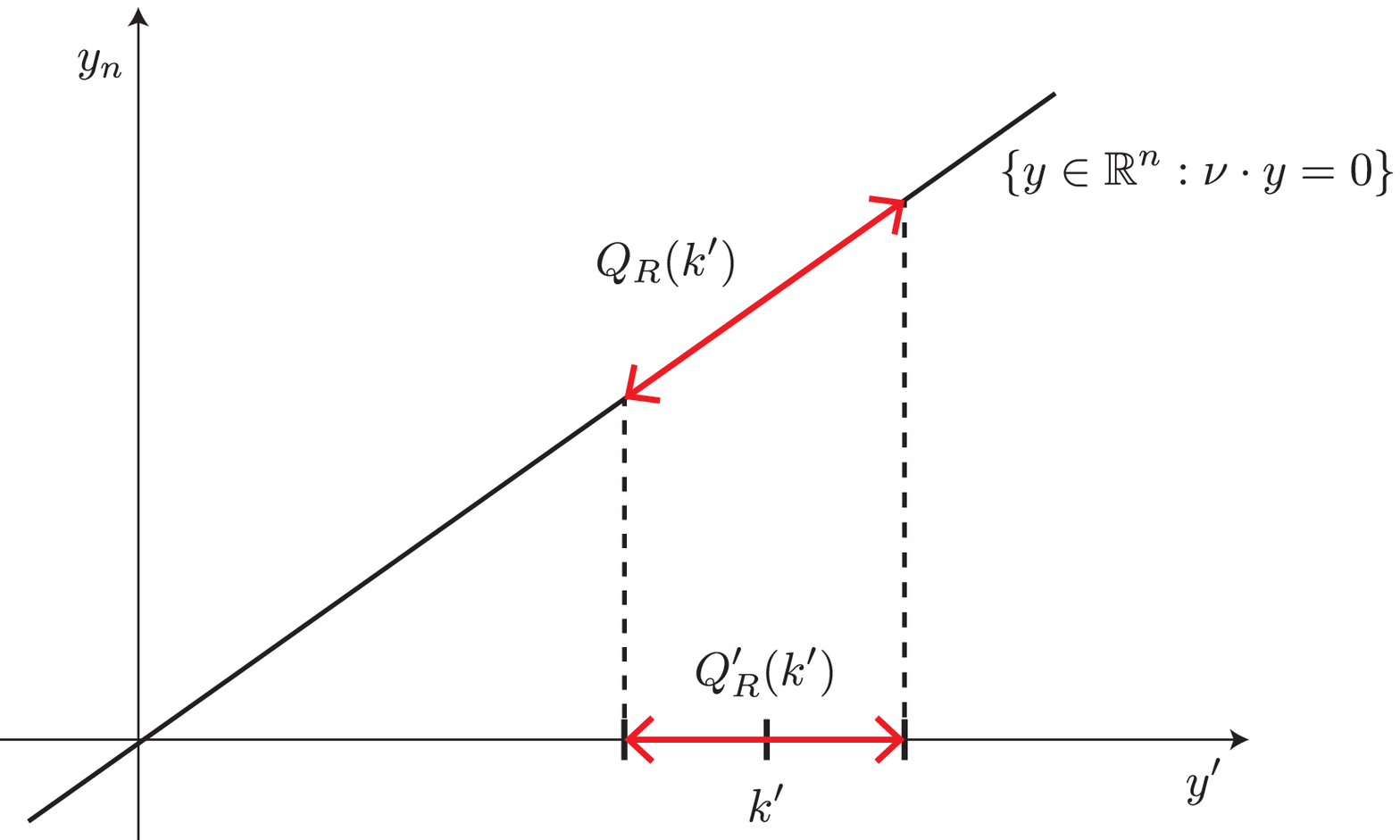}
\caption{The shape of $Q^\prime_R(k^\prime)$ and $Q_R(k^\prime)$.} 
\end{figure}

Suppose that $\nu_n \neq0$. Then we may think $\{ y \in \R^n : \nu \cdot y = 0 \}$ as a graph defined as 
\begin{equation}
y_n = -\frac{\nu_1}{\nu_n} y_1 - \cdots -\frac{\nu_{n-1}}{\nu_n} y_{n-1}.
\end{equation}
Let $Q^\prime_R$ be a cube in $\R^{n-1} \times \{0\}$ centered at $0$ with side length $R$ and $Q^\prime_R(k^\prime) = R k^\prime + Q^\prime_R$ for each $k^\prime \in \mathbb{Z}^{n-1} \times \{0\}$. We also let $Q_R(k^\prime)$ be a piece of $\{ y \in \R^n : \nu \cdot y = 0 \}$ whose projection to $\R^{n-1} \times \{0\}$ is $Q^\prime_R (k^\prime)$.

\begin{lemma} \label{hyper-ap}
Let $Q^\prime_R(k^\prime)$ and $Q_R(k^\prime)$ be given as the above for a direction $\nu \in S^{n-1}$. Assume that $\nu \in \mathcal{IR}$. Then, for any fixed $\delta >0$ and $k^\prime\in \mathbb{Z}^{n-1} \times \{0\}$, there exists a constant $R$, depending only on $\delta$, $\nu$ and a point $\hat y(k^\prime) \in Q_R(k^\prime)$ such that $|\hat y(k^\prime) - \widetilde m| \le \delta$ for some $\widetilde m \in \mathbb{Z}^n$. 
Moreover, for any given periodic function $g(y) \in \mathcal{C}^1(\R^n)$, we have
\begin{equation}
|g(y+\hat y(k^\prime)) -g(y)| \le \| \nabla g \|_{L^\infty(\R^n)} \delta 
\end{equation}
for all $y \in \R^n$. 
\end{lemma}

\begin{remark}
The above Lemma tells us that if $\nu$ is irrational, then a periodic function $g(y)$ is almost periodic on $\{ y \in \R^n : \nu \cdot y = \nu \cdot y_0 \}$ for all $y_0 \in \R^n$. This property is crucial to obtain a homogenized result.
\end{remark}

\begin{proof}
For simplicity, we assume that $\nu_n \neq 0$ and $\frac{\nu_1}{\nu_n}$ is an irrational number. Let $m \in \mathbb{Z}^{n-1} \times \{0\}$ be a point in $Q^\prime_R(k^\prime)$ and let $t=t(m) \in [0,1)$ be the fractional part of $h(m) = \frac{\nu_1}{\nu_n} m_1 + \cdots +\frac{\nu_{n-1}}{\nu_n} m_{n-1}$. 

Set $N_R = \#\{ \mathbb{Z}^{n-1} \times \{0\} \cap Q^\prime_R(k^\prime)\}$ and $A_R = \# \left\{ m \in \mathbb{Z}^{n-1} \times \{0\} \cap Q^\prime_R(k^\prime) : t(m) \in [0, \delta ) \right\}$.
From Lemma \ref{hyper-ud}, there is a modulus of continuity $\rho$ such that
\begin{equation}
A_R \approx \left(\delta + \rho\left(\frac{1}{R}\right) \right)N_R \approx \displaystyle\frac{1}{R^{n-1}}  \left(\delta + \rho\left(\frac{1}{R}\right) \right)
\end{equation}
for any given $\delta > 0$.
Therefore $A_R$ is nonempty if $R$ is large enough. Note that the modulus of continuity depends only on the direction $\nu$. Fix $R = R(\delta,\nu) >0$ such that $A_R$ becomes nonempty and then choose $m \in A_R$. Let $m \in A_R$, $\hat y(k^\prime) = \left( m, h(m)\right)$, and $\widetilde m = \left(m, h(m)-t(m) \right)$. Then, from the choice of $m$, we have
\begin{equation}
|t(\hat y)| = |\hat y(k^\prime) - \widetilde m|  \le \delta.
\end{equation}

Since $\widetilde m$ is an integer point and $g$ is periodic, we have
\begin{equation} \label{eq-hyper-ap}
g (y+ \hat y(k^\prime) ) = g ( y+ \hat y(k^\prime) - \widetilde m ) = g(y+t(m) e_n)
\end{equation}
for all $y \in \R^n$ and hence we conclude 
\begin{equation}
| g(y+\hat y(k^\prime)) -g(y) | = | g(y+t(m) e_n) - g(y) | \le \| \nabla g \|_{L^\infty(\R^n)} \delta .
\end{equation}
\end{proof}

Let us consider the following corrector equation,
\begin{equation} \label{eq-hyper-cor1} \begin{cases}
F(D^2 w,y) = 0 &\text{ in } H(\nu, y_0) \\
w(y) = g(y) &\text{ on } \partial H(\nu,y_0)
\end{cases} \end{equation}
where $H(\nu,y_0) = \{y \in \R^n : \nu \cdot y \ge \nu \cdot y_0 \}$. 

Via Perron's method in \cite{CIL}, we have the following:
\begin{lemma} \label{hyper-ex}
There is a viscosity solution of \eqref{eq-hyper-cor1} satisfying 
\begin{equation}
|w(y)| \le \|g\|_{L^\infty(\partial H(\nu,y_0))}.
\end{equation}
\end{lemma}

\begin{lemma} \label{hyper-uniq}
The solution of \eqref{eq-hyper-cor1} is unique. 
\end{lemma}
\begin{proof}
Suppose that there are two solutions $w_1$ and $w_2$ satisfying \eqref{eq-hyper-cor1}. From the Theorem \ref{thm-vis-diff}, $w_1-w_2 \in \mathcal{S}(\lambda, \Lambda)$. Moreover, since $w_1-w_2$ has zero boundary data on $\partial H(\nu,y_0)$, it should be zero because of the weak maximum principle in \cite{CLV}. 
\end{proof}

Let us introduce a regularity result for the solution of the equation \eqref{eq-hyper-cor1}. 
\begin{lemma} \label{hyper-reg1}
Suppose that $g(y)$ in \eqref{eq-hyper-cor1} is periodic, and in $\mathcal{C}^{2}(\square)$. Then we have 
\begin{equation}
\| w \|_{\mathcal{C}^{1,\alpha}( H(\nu,y_0))} \le C \| g \|_{\mathcal{C}^{2} (\square)}.
\end{equation}
Here $\square$ denotes the unit cell of $\R^n$ and $C$ is a constant depending only on $n$, $\lambda$, and $\Lambda$. 
\end{lemma}

From the maximum principle between $w$ and $\pm\| g\|_{L^\infty(\square)}$, we have $|w| \le \| g\|_{L^\infty(\square)}$. Then the proof of the Lemma above follows the similar argument in the proof of Lemma \ref{cor-reg1}, Lemma \ref{cor-reg2} and Lemma \ref{cor-reg3} in Section \ref{sec-cor}.

\begin{lemma} \label{hyper-reg2}
Let $w^i$, $i=1,2$, be the solutions of the equation \eqref{eq-hyper-cor1} for $y_0 = y_i \in \R^n$ respectively. Suppose that $g$ is a periodic $\mathcal{C}^2$ function. Then we have
\begin{equation}
| w^1(y) - w^2(y) | \le  C \| g \|_{\mathcal{C}^2(\square)} |y_1 -y_2| 
\end{equation}
for all $y \in H(\nu,y_1) \cap H(\nu,y_2)$ where $C$ is a constant depending only on $n$, $\lambda$, and $\Lambda$. 
\end{lemma} 
\begin{proof}
Without any loss of generality, we may assume that $y_2 = 0$ and $H(\nu, y_1) \subset H(\nu,y_2)$. Note that for given any $y \in \partial H(\nu,0)$, $y+y_1 \in \partial H(\nu,y_1)$. From Lemma \ref{hyper-reg1}, we have 
\begin{equation}
| w^2(y+y_1) - g(y) | =| w^2(y+y_1) - w^2(y) | =  C \| g \|_{\mathcal{C}^2(\square)} |y_1|
\end{equation}
where $y \in \partial H(\nu,0)$ and $C$ is a constant same as the Lemma \ref{hyper-reg1}. 
Moreover, since $g$ is in $\mathcal{C}^2$, we have
\begin{equation}
| g(y+y_1) - g(y) | \le \| g \|_{\mathcal{C}^2(\square)} |y_1|. 
\end{equation}
Now combining two inequalities above, we have
\begin{equation}
| w^2(y + y_1) -  g(y + y_1) | \le C \| g \|_{\mathcal{C}^2(\square)} |y_1|
\end{equation}
for every $y \in \partial H(\nu,0)$. 
Now from the Theorem \ref{thm-vis-diff} and the maximum principle in \cite{CLV}, we can conclude that 
\begin{equation}
| w^2(y) - w^1(y) | \le \| w^2 - w^1 \|_{L^\infty(\partial H(\nu,y_1))} \le  C \| g \|_{\mathcal{C}^2(\square)} |y_1| 
\end{equation}
for all $y \in H(\nu,y_1)$. 
\end{proof}

For simplicity, we denote $\Pi = \partial H(\nu,y_0)$ and $\Pi(t) = t \nu + \Pi$ for positive real number $t$ until the end of this section.
\begin{lemma} \label{hyper-osc}
Let $w$ solve the equation \eqref{eq-hyper-cor1}. Suppose that $\nu$ in the equation \eqref{eq-hyper-cor1} is irrational. Then we have 
\begin{equation}
\lim_{t \rightarrow \infty} W_t = 0
\end{equation}
where $W_t = \osc_{\Pi(t)} w = \sup_{y_1,y_2 \in \Pi(t)} |w(y_1) - w(y_2)|$.
\end{lemma}
\begin{proof}
We only prove the case when $y_0 = 0$ because the general case can be obtained by the translation. Let $Q^\prime_R(k^\prime)$ and $Q_R(k^\prime)$ be the same as Lemma \ref{hyper-ap}. Fix $\delta > 0$. Since $\nu$ is irrational, we may choose $R>0$ such that each cube $Q_R(k^\prime) \subset \partial H(\nu,0)$ has a point $\hat y=\hat y(k^\prime)$ satisfying 
\begin{enumerate}
\item
$\hat y = (m, h(m))$ for some $m \in \mathbb{Z}^{n-1} \cap Q^\prime_R(k^\prime)$, 
\item
the fractional part $t(m)$ of $h(m)=\frac{\nu_1}{\nu_n} m_1 + \cdots + \frac{\nu_{n-1}}{\nu_n} m_{n-1}$ is less than or equal to $\delta$
\end{enumerate}
from Lemma \ref{hyper-ap}.

From the definition of $W_t$, we can choose points $y_1$ and $y_2$ in $\Pi(t)$ such that 
\begin{equation}
W_t \le \left|w(y_1) - w(y_2) \right| + \delta
\end{equation}
for given any $\delta>0$. 

We denote $y_i = y^\prime_i+t\nu$, $y^\prime_i \in \partial H(\nu,0)$ for $i=1,2$. Without any loss of generality, we can assume that $y^\prime_1 \in Q_R(0)$ and $y^\prime_2 \in Q_R(k_0^\prime)$ for fixed $k^\prime_0$. 

Note that $\hat y = \hat y(k_0^\prime)$ consists of two parts, the integer part $\widetilde m = (m,h(m)-t(m)) \in \mathbb{Z}^n$ and the fractional part $t(m) e_n$, $0 \le t(m) \le \delta$. Note also that $y^\prime_2 -\hat y \in \partial H(\nu,0)$ is contained in the cube $Q_{3R}(0)$.

Let $\widetilde w(y) = w (\widetilde y) = w(y - \widetilde m)$. Then, from the relation $\widetilde y = y-\widetilde m = y - \hat y + t(m) e_n \in H(\nu, y_0)$ and the periodicity of $F$ and $g$, $\hat w$ is a solution of the following equation,
\begin{equation} \label{eq-hyper-cor2} \begin{cases}
F(D^2 \widetilde w,y) = 0 &\text{ in } H(\nu, y_0 - t(m)e_n) \\
\widetilde w(y) = g(y) &\text{ on } \partial H(\nu,y_0 - t(m)e_n).
\end{cases} \end{equation}

So, from Lemma \ref{hyper-reg2}, we have 
\begin{equation}
|\widetilde w(y) - w(y) | \le C\|g\|_{\mathcal{C}^2} \delta
\end{equation}
for all $y \in H(\nu,y_0)$. 
In particular, we have
\begin{equation}
|w(y_2 - \widetilde m) - w(y_2) | \le C\|g\|_{\mathcal{C}^2} \delta.
\end{equation}

Note that the constant $C$ represents a constant that only depends on $n$, $\lambda$ and $\Lambda$. That constant $C$ could change as the equation changes, but the dependences does not change at least in this proof. 

From the Lemma \ref{hyper-reg1}, we have $|w(y_2 - \widetilde m) - w(y_2 - \hat y) | \le C\|g\|_{\mathcal{C}^2} \delta$ and hence we can conclude that 
\begin{equation}
|w(y_2) - w(y_2 - \hat y) | \le C\|g\|_{\mathcal{C}^2} \delta.
\end{equation} 

Since $y_1$ and $y_2 -\hat y$ contained in a cube $Q_{3R}(t) = Q_{3R} + t \nu$, we have 
\begin{equation} \label{eq-hyper-osc2}\begin{aligned}
W_t &\le \left|w(y_1) - w(y_2) \right| + \delta \\
&\le \left|w(y_1) - w(y_2 - \hat y) \right| +  C \|g\|_{\mathcal{C}^2} \delta \\
&\le \osc_{Q_{3R}(t)} w +  C \|g\|_{\mathcal{C}^2} \delta.
\end{aligned} \end{equation} 

Suppose that $2^m (3R) \le t$ for some $m \in \mathbb{Z}$. Then, by applying the first property in the Proposition \ref{prop-vis-osc-holder} $m$ times, we have
\begin{equation} \begin{aligned}
\osc_{Q_{3R}(t)} w &\le \osc_{B_{3R}(te_n)} w \le \gamma^{1} \osc_{B_{2 \times (3R)}(te_n)} w \\
&\le \cdots \le \gamma^m \osc_{B_{2^m \times (3R)}(te_n)} w \\
&\le \gamma^m \osc_{B_{t}(te_n)} w \\
&\le 2 \gamma^m |w|_\infty \\
\end{aligned} \end{equation}
where $\gamma$ is a constant in $(0,1)$ depending only on the dimension.
Hence we have 
\begin{equation} \label{eq-hyper-osc3}
\osc_{Q_{3R}(t)} w \ \le 2 \left( \displaystyle\frac{3R}{t}\right)^{\log_2 \frac{1}{\gamma}} |w|_\infty.
\end{equation}

Now by using \eqref{eq-hyper-osc2} and \eqref{eq-hyper-osc3}, we have the following,
\begin{equation} \label{eq-hyper-cor} \begin{aligned}
0 &\le W_t \le |w(y_1) - w(y_2)| + \delta \\
&\le \osc_{Q_{3R}(t)} w +  C \|g\|_{\mathcal{C}^2} \delta\\
&\le 2 \left( \displaystyle\frac{3R}{t}\right)^{\log_2 \frac{1}{\gamma}} |w|_\infty + C \|g\|_{\mathcal{C}^2} \delta
\end{aligned} \end{equation}
for sufficiently large $t > 0$.

By taking limit infimum and supremum of $W_t$, we get
\begin{equation}
0 \le \liminf_{t \rightarrow \infty} W_t \le \limsup_{t \rightarrow \infty} W_t \le \lim_{t \rightarrow \infty} \left( \displaystyle\frac{3R}{t}\right)^{\log_2 \frac{1}{\gamma}} |w|_\infty + C \|g\|_{\mathcal{C}^2}  \delta = C \|g\|_{\mathcal{C}^2} \delta.
\end{equation}
Since $\delta$ is arbitrary, we get the conclusion. 
\end{proof}

\begin{lemma} \label{hyper-limit}
Let $w(y)$ be the solution of \eqref{eq-hyper-cor2} and $\nu$ in \eqref{eq-hyper-cor1} is irrational. Then, the limit $w(y^\prime +t p)$ exists as $t$ goes to $\infty$ for each $y^\prime \in \Pi$ and $p \in \R^n$ satisfying $p \cdot \nu >0$. Moreover, that limit is independent of the choice of $y^\prime$ and $p$.
\end{lemma}
\begin{proof}
We assume that $\nu = e_n$ and $y_0=0$ and the result for general case can be obtained by the rotation and translation. We also assume that $p \cdot \nu =1$. Let $M_t = \sup_{\Pi(t)} w$ and let $m_t = \inf_{\Pi(t)} w$. Then $W_t$ is given by $M_t -m_t$ and $m_t \le w(y^\prime +t p) \le M_t$. 

Since $w(y)$ satisfies the following equation,
\begin{equation} \begin{cases}
F(D^2 w,y) = 0 &\text{ in } \R^{n-1} \times \{ y_n > t \}, \\
w= w(y^\prime +te_n) &\text{ on } \partial ( \R^{n-1} \times \{ y_n > t \} ), 
\end{cases} \end{equation} 
$M_t = \sup_{\R^{n-1} \times \{ y_n > t \} } w$ and $m_t = \inf_{\R^{n-1} \times \{ y_n > t \} } w$ from the weak maximum principle on the unbounded domain $\R^{n-1} \times \{ y_n > t \}$ (see \cite{CLV}). It implies that $M_t$ is monotone decreasing and $m_t$ is monotone increasing and hence there exist $\alpha^* = \lim_{t \rightarrow \infty} M_t$ and $\alpha_* = \lim_{t \rightarrow \infty} m_t$. From lemma \ref{hyper-osc}, $\alpha^*$ and $\alpha_*$ have to be the same and that should be equal to the limit of $w(y^\prime +t p)$. 
\end{proof}

\section{Correctors}\label{sec-cor}
In this section, we are going to consider the corrector equation defined as follow,
\begin{equation} \label{eq-cor1} \begin{cases}
F( D^2 w_\e, y) = 0 &\text{ in } H \left( \nu, y_{0,\e} \right) \\
w_\e = g( x_0, y ) &\text{ on } \partial H\left( \nu, y_{0,\e} \right)
\end{cases} \end{equation}
where $ x_0 \in \partial D$, $y_{0,\e} = x_0 / \e$, and $\nu \in S^{n-1}$. 

\begin{definition} \label{def-cor-overg}
Let $w_\e$ be the solution of the equation \eqref{eq-cor1}. Then we denote
\begin{equation} \begin{aligned}
\overline g^*(x_0,\nu) &= \limsup_{\e \rightarrow 0} \limsup_{t \rightarrow \infty} w_\e(y_{0,\e} + t \nu), \\
\overline g_*(x_0,\nu) &= \liminf_{\e \rightarrow 0} \liminf_{t \rightarrow \infty} w_\e(y_{0,\e} + t \nu).
\end{aligned}\end{equation} 

If $\overline g^*$ and $\overline g_*$ are the same, then we define
\begin{equation}
\overline g(x_0,\nu) = \overline g^*(x_0,\nu).
\end{equation}
\end{definition} 

\begin{example} \label{ex-cor-overg}
Choose $x_0 = (t,1) \in \R^2$, $-1 < t < 1$ and $\nu = e_2$. Assume that $g(y_1,y_2) = \cos (\pi y_2)$. 
If we select a subsequence $\e_m = \displaystyle\frac{1}{2m}$, then $y_{0,\e_m} = (2m t, 2m)$ and hence $g(y) = 1$ for all $y\in H \left( \nu, y_{0,\e} \right)$. This implies that $1$ is a solution of the equation \eqref{eq-cor1}. So, we have
\begin{equation}
\lim_{m \rightarrow \infty} w_{\e_m} (y_{0,\e_m} + t \nu) = 1.
\end{equation}
Since $|w_\e| \le 1$, we have $\overline g^*(x_0,\nu)=1$. In this way, we can show that $\overline g_* = -1$ by choosing $\e_m = \displaystyle\frac{1}{2m+1}$. So $\overline g^* \neq \overline g_*$ in this case. 
\end{example}

We always observe the above phenomena when $\nu$ is not irrational because the hyper plain $\partial H(\nu, y_0) / \mathbb{Z}^n $ is not uniformly distributed in $[0,1]^n$. However, $\overline g^*(x_0,\nu) = \overline g_*(x_0,\nu)$ if $\nu$ is irrational. 

Let $F_\e(M,y) = F(M, y+y_{0,\e})$ and $g_\e(x_0,y) = g (x_0, y+y_{0,\e})$. Then $\widetilde{w}_\e(y) = w_\e(y + y_{0,\e})$ is a solution of the following equation,
\begin{equation} \label{eq-cor2} \begin{cases}
F_\e( D^2 w_\e,y) = 0 &\text{ in } H \left( \nu, 0\right) \\
w_\e = g_\e (x_0, y) &\text{ on } \partial H\left( \nu, 0\right).
\end{cases} \end{equation}

Note that the estimate \eqref{eq-hyper-cor} depends only on $n$, $\lambda$, $\Lambda$, and $\|g\|_{\mathcal{C}^2}$. So we have the following uniform oscillation bounds,
\begin{equation} \label{eq-cor-osc}
W_{\e,t} \le \inf_{\delta >0} \left[ 2 \left( \displaystyle\frac{3R(\delta)}{t} \right)^{-\log_2 \gamma} |g|_\infty +  C \|g\|_{\mathcal{C}^2}\delta \right].
\end{equation}
Hence the oscillation $W_{\e,t} = \osc_{\Pi(t)} \widetilde{w}_\e$ goes to zero uniformly on $\e$, and $\alpha_\e =\lim_{t \rightarrow \infty} \widetilde{w}_\e(t \nu) $ is well defined for each $\e$. 
\begin{lemma} \label{cor-lim}
If $\nu$ is a irrational direction, then the limit $\alpha_\e$ in the above are independent on $\e$. In other words, $\overline g(x_0,\nu)$ is well defined and $\alpha_\e = \overline g(x_0,\nu)$ for all $\e>0$. 
\end{lemma} 
\begin{proof}
Suppose that $\widetilde w_1$ is a solution of the equation \eqref{eq-cor1} when $\e =\e_1$ and $w_2$ is a solution of the equation \eqref{eq-cor2} when $\e=\e_2$. 
By translating $\widetilde w_1$ properly, we may assume that $w_1$ is a solution of the equation,
\begin{equation} \begin{cases}
F_{\e_2}( D^2 w_1,y) = 0 &\text{ in } H \left( \nu, y_{0,\e_1} - y_{0,\e_2} \right) \\
w_1 = g_{\e_2} (x_0,y) &\text{ on } \partial  H \left( \nu, y_{0,\e_1} - y_{0,\e_2} \right).
\end{cases} \end{equation}

Since $\nu$ is irrational, by using the similar argument in the proof of Lemma \ref{hyper-osc}, we can find $z = -\widetilde m + t e_n \in \partial  H \left( \nu, y_{0,\e_1} - y_{0,\e_2} \right)$ where $\widetilde m \in \mathbb Z^n$, $ t \in [0,\delta)$ for given any $\delta >0$ and $e_n$ is the $n$-th member of the standard coordinate basis of $\R^n$.  

Note that if $y - \widetilde m \in H(\nu, y_{0,\e_1} - y_{0,\e_2})$ then $y \in H(\nu, y_{0,\e_1} - y_{0,\e_2} + \widetilde m)$ and $H(\nu, y_{0,\e_1} - y_{0,\e_2} + \widetilde m) = H(\nu, z + \widetilde m) = H(\nu, t e_n)$. Hence $\widehat w(y) = w_1(y-\widetilde m)$. satisfies the following equation,
\begin{equation} \begin{cases}
F_{\e_2}( D^2 \widehat w,y) = 0 &\text{ in } H \left( \nu, t e_n \right) \\
\widehat w = g_{\e_2} (x_0,y) &\text{ on } \partial H \left( \nu, t e_n \right)
\end{cases} \end{equation}

Now apply Lemma \ref{hyper-reg2} to obtain 
\begin{equation}
| \widehat w(y) - w_2(y) | \le C\| g(x_0,\cdot)\|_{\mathcal{C}^2(\square)} \delta
\end{equation}
if $y \in H(\nu,0) \cap H(\nu, t e_n)$ where $C$ is a constant depending only on $n$, $\lambda$ and $\Lambda$. 

Choose $s$ large enough and substitute $y = s \nu$. Then we have
\begin{equation} \label{eq-cor-lim1}
| w_1 ( s \nu - \widetilde m ) - w_2( s \nu ) | \le C\| g(x_0,\cdot)\|_{\mathcal{C}^2(\square)} \delta. 
\end{equation}

Because of Lemma \ref{hyper-limit}, $w_1 ( s \nu - \widetilde m ) \rightarrow \alpha_{\e_1}$ as $s \rightarrow \infty$. So we can have the  following by taking $\lim_{s \rightarrow \infty}$ on both side to the equation \eqref{eq-cor-lim1},
\begin{equation}
| \alpha_{\e_1} - \alpha_{\e_2} | \le  C\| g(x_0,\cdot)\|_{\mathcal{C}^2(\square)} \delta. 
\end{equation}

Since $\delta > 0$ can be chosen arbitrarily small, we get $\alpha_{\e_1} = \alpha_{\e_2}$ and this exactly implies the conclusion.
\end{proof}

If $F$ is the Laplace operator, then we can describe the homogenized operator $\overline g$ by the average of $g$. So, we can recover the result in \cite{LS} for the Laplace operator through our method.
\begin{prop} \label{prop-cor-linear}
Suppose that $F(M,\frac{x}{\e}) =  a_{ij} M_{ij}$ for some constant matrix $(a_{ij})$. Then, $\overline g(x,\nu) = \langle g \rangle(x) = \int_{[0,1]^n} g(x,y) dy$ for all $\nu \in \mathcal{IR}$.  
\end{prop}
\begin{proof}
Fix $x_0 \in \partial D$ and $\e >0$. Assume that $g_\e(x_0,y)$ and $H(\nu,0)$ are defined the same in \eqref{eq-cor2} and $w_\e(y)$ is the solution of the equation \eqref{eq-cor2}.

Let $Q_R(y)$ be a cube in $\partial H(\nu,0)$ centered at $y \in \partial H(\nu,0)$ with side length $R$. Suppose that $Q_R(y)$ is divided into disjoint cubes $\{ Q^i \}$ with side length $\delta>0$ and $y^i  \in Q_R(y)$ are centers of $ Q^i$. Let us define 
\begin{equation} \label{eq-cor-linear1}
G_R^\delta(y) = \displaystyle\frac{1}{R^{n-1}} \sum_{i} g_\e(x_0,y^i) \delta^{n-1} \quad\text{  and  }\quad G_R(y) = \fint_{Q_R(y)} g_\e(x_0,y) d\sigma_y.
\end{equation}
Since $g_\e(x_0,\cdot)$ is uniformly continuous on the $y$-variable, $G_R^\delta(y)$ converges to $G_R(y)$ uniformly on $y \in \partial H(\nu,0)$ as $\delta \rightarrow 0$. 

Since $F$ is independent on the $y$ variable and linear, $W_R^\delta(y + t\nu) = \displaystyle\frac{1}{R^{n-1}} \sum_i w_\e(y^i + t\nu) \delta^{n-1}$ is also a solution of the equation \eqref{eq-cor2} with boundary condition $G_R^\delta(y)$ where $y \in \partial H(\nu,0)$ and $t >0$. Moreover, since $G^\delta_R(y)$ converges to $G_R(y)$ uniformly, $W_R^\delta$ converges to $W_R(y)$, the solution of \eqref{eq-cor2} when the boundary condition is $G_R(y)$, uniformly and satisfying
\begin{equation} \label{eq-cor-linear2}
\| W_R^\delta -W_R \|_{L^\infty(H(\nu,0))}  \le \|G_R^\delta - G_R \|_{L^\infty(\partial H(\nu,0))} 
\end{equation}
because of the maximum principle in \cite{CLV}.

Since $\nu$ is irrational, from Lemma \ref{hyper-limit}, $\alpha_\e = \lim_{t \rightarrow \infty} w_\e(y+t\nu)$ is well defined and independent of the choice of $y \in \partial H(\nu,0)$. Moreover, $W_R^\delta(x + t\nu)$ is just a finite sum of $w_\e(y^i + t\nu)$, $\alpha_\e = \lim_{t \rightarrow \infty}  W_R^\delta(x + t\nu) $. 

Finally, from \eqref{eq-cor-linear2} and the uniform convergence of $G^\delta_R$, we have
\begin{equation} \begin{aligned}
|\alpha_\e - &\lim_{t \rightarrow \infty}  W_R(y + t\nu)| = \lim_{t \rightarrow \infty} | \alpha_\e -W_R(y + t\nu)| \\
&\le \lim_{t \rightarrow \infty}\left( | \alpha_\e -W^\delta_R(y + t\nu)|  + |W^\delta_R(y + t\nu) - W_R(y + t\nu) | \right) \\
&\le \lim_{t \rightarrow \infty} | \alpha_\e -W^\delta_R(y + t\nu)|  + |G^\delta_R(y) - G_R(y) |_\infty \\
&\le |G^\delta_R(y) - G_R(y) |_\infty \\
&\rightarrow 0
\end{aligned} \end{equation}
as $\delta \rightarrow 0$ and hence
\begin{equation}
\alpha_\e = \lim_{t \rightarrow \infty}  W_R(y + t\nu).
\end{equation}

Assume that $G_R(y)$ converges to $\langle g\rangle$ uniformly. Then by using similar argument in the above, we can show that $\alpha_\e = \lim_{t \rightarrow \infty}  W (y + t\nu) $ where $W$ is the solution of \eqref{eq-cor2} when the boundary condition is a constant function $\langle g \rangle(x_0)$ and since the boundary data is a constant, $W(y)$ should be equal to $\langle g\rangle(x_0)$ in $H(\nu,0)$. Hence we have $\alpha_\e = \langle g \rangle(x_0)$ for all $\e>0$ and $\overline g(x_0,\nu) = \langle g \rangle (x_0)$. 

Now we will prove the uniform convergence of $G_R(y)$. For the simplicity, we assume $|g|_\infty + |\nabla g|_\infty \le 1$. For fixed $\delta >0$, we can choose $R_0>0$ such that the plane $\partial H(\nu,0)$ is represented as a union of disjoint cubes $Q_{R_0}(k^\prime)$ defined same in Lemma \ref{hyper-ap} and each cube has a point $\hat y(k^\prime)$ such that $|g(y + \hat y(k^\prime)) -g(y)| \le \delta$ for all $y \in \partial H(\nu,0)$. 

Suppose that $y_1$ and $y_2$ are in $\partial H(\nu,0)$ and $|y_1-y_2|\le 3nR_0 \le R$. Then, 
\begin{equation} \label{eq-cor-linear3} \begin{aligned}
|G_R(y_1) - G_R(y_2)| &\le \left| \fint_{Q_R(y_1)} g(y) d\sigma - \fint_{Q_R(y_2)} g(y) d\sigma \right| \\
&\le \displaystyle\frac{2|Q_R(y_1) \setminus Q_R(y_2)|_{n-1}}{R^{n-1}} \\
&\le \displaystyle\frac{c(n) R^{n-2} R_0}{R^{n-1}} \le \displaystyle\frac{c(n)R_0}{R} \\
\end{aligned} \end{equation}
where $|\cdot|_{n-1}$ is a measure on $\partial H(\nu,0)$ reduced by the Lebesgue measure in $\R^n$ and $c(n)$ is a constant depending only on the dimension $n$.

Choose $y_1$, $y_2$ in $\partial H(\nu,0)$. Without any loss of generality, we may assume that $y_1 \in Q_{R_0}(0)$ and $y_2 \in Q_{R_0}(k^\prime)$ for some $k^\prime \in \mathbb{Z}^{n-1}$. From the choice of $R_0$, we can find $\hat y \in Q_{R_0}(k^\prime)$ such that $|g(y+ \hat y) - g(y)|\le \delta$ and, in particular, $|g(y+y_1+ \hat y) - g(y+y_1)|\le \delta$ for all $y \in \partial H(\nu,0)$. 
Moreover, since $y_1 \in Q_{R_0}(0)$ and $\hat y \in Q_{R_0}(k^\prime)$, we have
\begin{equation}
|y_1+\hat y - y_2 | \le 3n R_0.
\end{equation}

Combining those, we have the following,
\begin{equation} \begin{aligned}
|G_R(y_1&) - G_R(y_2)| \le |G_R(y_1) - G_R(y_1 + \hat y)| + | G_R(y_1 +\hat y) - G_R(y_2) | \\
&\le \left| \fint_{Q_R(y_1)} g(y) d\sigma - \fint_{Q_R(y_1+\hat y)} g(y) d\sigma \right| + | G_R(y_1 +\hat y) - G_R(y_2) | \\
&\le\fint_{Q_R(y_1)} \left|  g(y) - g(y + \hat y) \right| d\sigma  + \frac{c(n)R_0}{R} \\
&\le \delta +\frac{c(n)R_0}{R}
\end{aligned} \end{equation}

Hence, if $R$ is large enough, then $|G_R(y_1) - G_R(y_2)| \le 2 \delta $. Finally, since $G_R(0)$ converges to $\langle g \rangle(x_0)$, we can conclude that
\begin{equation}
|G_R(y) - \langle g \rangle(x_0) | \le |G_R(y) - G_R(0) | + |G_R(0) - \langle g \rangle(x_0) | \le 3 \delta
\end{equation}
when $R >0$ is sufficiently large and since the choice of $R$ is depending only on $R_0$ and the convergence speed of $G_R(0)$, $G_R(y)$ converges to $\langle g \rangle(x_0)$ uniformly on $y \in \partial H(\nu,0)$. 
\end{proof}

\begin{remark} \label{rmk-cor-linear}
The Proposition \ref{prop-cor-linear} hold if $y = (z_1, z_2)$, $z_1 \in \R^m$ and $z_2 \in \R^{n-m}$, $a_{ij}$ is a function on the variable $z_1$ and $g$ is a function on the $z_2$ variable. The main idea is similar to the proof of Proposition \ref{prop-cor-linear}. The only differencs is that we integrate the solution along $z_2$ axis. 
\end{remark}

We omit $\nu$ of $\overline g(x_0, \nu)$ if $x_0 \in \partial D$ and $\nu$ is the same as the inward normal vector $-\nu(x_0)$ of $D$. In other words, $\overline g(x_0) = \overline g(x_0,-\nu(x_0))$. As we proved, if $\nu$ is irrational, then $\overline g(x_0)$ is well defined. 

\begin{definition} \label{def-cor-iddc}
The domain $D$ satisfies the irrational direction dense condition(IDDC) if the size of the set $\{ x \in \partial D | \nu(x) \in \mathcal R \}$ is countable.
\end{definition}

As we told in the introduction, there are lots of examples satisfying the IDDC. $\mathcal{C}^2$ domains whose boundary has nonzero Gaussian curvature, balls as a example, satisfies IDDC.


Now we are going to show that $u_\e$ gets closer to $w_\e$ after blow-up near the $\mathcal C^2$-boundary $\partial D$. Note that the scaled function $U_\e(y) = u_\e(\e y)$ satisfies the following equation
\begin{equation} \label{eq-cor-approx} \begin{cases}
F(D^2 U_\e,y) = \e^2 f(\e y,y) &\text{ in } \e^{-1} D \\
U_\e(y) = g(\e y, y) &\text{ on } \partial  \left( \e^{-1} D \right) \\
\end{cases} \end{equation}
from \eqref{con-phomo}.

We will show that the uniform boundedness of $\| U_\e (\cdot) \|_{\mathcal{C}^{1,\alpha}(\e^{-1} \overline  D)}$. First, let us show the uniform boundedness of $\| U_\e(\cdot)\|_{\mathcal C^{1,\alpha}(\e^{-1}D)}$ by summarizing known results.
\begin{lemma} \label{cor-reg1}
Suppose that $v$ is a viscosity solution of 
\begin{equation} \begin{cases}
F(D^2 v ,x) = f(x) &\text{ in } B^+_1(0) \\
v=g &\text{ on } \partial B^+_1(0)
\end{cases} \end{equation}
where $B^+_1(0)$ is a half ball and $g \in \mathcal{C}^2(\overline B^+_1(0))$. 
Then, we have
\begin{equation}
\| v \|_{\mathcal{C}^{1,\alpha}(B^+_{1/2}(0))} \le C\left(\|f\|_{L^\infty(B^+_1(0))} + \|g\|_{\mathcal{C}^2(B^+_1(0))} \right)
\end{equation}
where $C$ depends only on $n$, $\lambda$, and $\Lambda$.
\end{lemma} 

One can find the above boundary estimate of viscosity solution at \cite{LT} or it can be proved directly by the odd extension of the function $\widetilde v = v - g$ on $B_1(0)$ and the interior estimate of the viscosity solution in \cite{CC}. 

Suppose that $\partial D$ is in $\mathcal{C}^2$. Then, by using the domain straitening map $\Phi$ in the appendix of \cite{E}, we have the following,
\begin{lemma} \label{cor-reg2}
Let $v$ be a viscosity solution of the following equation. 
\begin{equation} \begin{cases}
F(D^2 v ,x) = f(x) &\text{ in } D\\
v(x) = g(x) &\text{ on } \partial D.
\end{cases} \end{equation}
Suppose that the $\mathcal{C}^2$ norm of $\partial D$ is sufficiently small. 

Then there is a constant $0<r_0<1$ and for $x\in \partial D$ we have
\begin{equation}
\| v \|_{\mathcal{C}^{1,\alpha}(B_{r_0}(x)) \cap \overline D)} \le C\left(\|f\|_{L^\infty(B_1(x) \cap D)} + \|g\|_{\mathcal{C}^2(B_1(x) \cap \partial D)} + \|v\|_{L^\infty(\overline B_1(x) \cap D)} \right).
\end{equation}
Moreover, the constants $r_0$ and $C$ only depends on $n$, $\lambda$, $\Lambda$ and the $\mathcal{C}^2$ norm of $\partial D$. 
\end{lemma}
Since we exactly know the formula $\Phi$, we can find such a $r_0$ by the calculation. 

From the above lemma, the interior estimate of the viscosity solution in \cite{CC} and the proper covering of the domain, we have
\begin{lemma} \label{cor-reg3}
Suppose that $v$ is solution of 
\begin{equation} \begin{cases}
F(D^2 v ,x) = f(x) &\text{ in } D\\
v(x)=g(x) &\text{ on } \partial D,
\end{cases} \end{equation}
and the $\mathcal{C}^2$ norm of $\partial D$ is sufficiently small. 
Then, $v \in \mathcal{C}^{1.\alpha}(\overline D)$ and we have
\begin{equation}
\| v \|_{\mathcal{C}^{1,\alpha}(\overline D)} \le C\left(\|f\|_{L^\infty(D)} + \|g\|_{\mathcal{C}^2(D)} + \|v\|_{L^\infty(D)} \right)
\end{equation}
where $C$ depends only on $n$, $\lambda$, $\Lambda$ and the $\mathcal{C}^2$ norm of $\partial D$. 
\end{lemma}

Since $|u_\e| \le C \|f\|_\infty + \|g\|_\infty$, $|U_\e|$ is also bounded by the same constant $C \|f\|_\infty + \|g\|_\infty$ where $C$ is a constant depending only on $n$, $\lambda$, $\Lambda$ and the diameter of $D$. Now let us apply Lemma \ref{cor-reg3} to the solution $U_\e$ of equation \eqref{eq-cor-approx} to obtain 
\begin{equation}
\| U_\e \|_{\mathcal{C}^{1,\alpha}(\e^{-1} \overline D)} \le C\left(\| f(x,y)\|_{L^\infty(D \times \R^n)} + \|g(x,y)\|_{\mathcal{C}^2(D \times \R^n)} \right).
\end{equation}
Note that the constant $C$ only depends on $n$, $\lambda$, $\Lambda$, and $D$. In this way, we can show the uniform boundedness of $w_\e$.
\begin{lemma} \label{cor-reg}
Suppose that $U_\e$ is the solution of the equation \eqref{eq-cor-approx} and $w_\e$ is a solution of the equation \eqref{eq-cor1}. Suppose also that $\e$ is sufficiently small. Then we have the following estimate,
\begin{equation} \label{eq-cor-reg}
\| U_\e \|_{\mathcal{C}^{1,\alpha} \left(\e^{-1} D\right)} + \| w_\e \|_{\mathcal{C}^{1,\alpha} \left( H(\nu,y_{0,\e}) \right)}  \le C\left(\| f(x,y)\|_{L^\infty(D \times \R^n)} + \|g(x,y)\|_{\mathcal{C}^2(D \times \R^n)} \right)
\end{equation}
where $C$ is a constant which depends only on $n$, $\lambda$, $\Lambda$ and $D$.
\end{lemma}

For given $x_0 \in \partial D$, let $y_{0,\e} = \frac{x_0}{\e}$ and $H(\nu,y_{0,\e}) = \{y \in \R^n : y \cdot \nu > y_{0,\e} \cdot \nu  \}$ where $\nu=-\nu(x_0) $ is the inward normal vector at $x_0$. Consider the small ball $B_{\varepsilon^p}(x_0)$, $0<p<1$. Let $w_\e$ be the solution of \eqref{eq-cor1}. Then the difference between $U_\varepsilon$ and $w_\varepsilon$ on $ \partial  \left( \e^{-1} D \right) \cap B_{\e^{p-1}} \left( y_{0,\e} \right)$ vanishes  as $\varepsilon \rightarrow 0$ if $2p >1$. 
\begin{lemma} \label{cor-approx1}
Let $U_\e$ be a solution of Equation \eqref{eq-cor-approx}, $x_0 \in \partial D$, and let $w_\e$ is a solution of \eqref{eq-cor1} for given fixed $x_0$ and $\nu=-\nu(x_0)$. Then, we have
\begin{equation}
\left| U_\e(y) - w_\e(y) \right| \le C \e^{2p-1} \quad\text{ on } \partial  \left( \e^{-1} D \cap H \left(\nu, y_{0,\e} \right) \right) \cap B_{\e^{p-1}} \left( y_{0,\e} \right)
\end{equation}
for all $y \in \partial  \left( \e^{-1} D \right) \cap B_{\e^{p-1}} \left( y_{0,\e} \right)$ and for all $0<p<1$ where $C$ depends only on $n$, $\lambda$, $\Lambda$ and the $\mathcal{C}^2$ norm of $\partial D$. 
\end{lemma}
\begin{proof}
For the simplicity of the calculation, we assume that $\nu=e_n$ and the general result can be obtained by a rotation. 
Since $\partial D$ is in $\mathcal{C}^2$, $\partial D \cap B_{\e^{p-1}}\left( y_{0,\e} \right)$ is contained between two hyperplanes $C_1 \e^{2p-1} \nu + \partial H(-\nu,y_{0,\e})$ and $C_1 \e^{2p-1} \nu - \partial H(-\nu,y_{0,\e})$ where $C_1$ is the constant depends only on the $\mathcal{C}^2$ norm of the boundary $\partial D$.
Fix $y \in \partial  \left( \e^{-1} D \cap H \left(-\nu,y_{0,\e} \right) \right) \cap B_{\e^{p-1}} \left( y_{0,\e} \right)$. We first assume that $y \in \partial \left( \e^{-1} D \right)$. Then, from the estimate \eqref{eq-cor-reg}, we have
\begin{equation} \begin{aligned}
\left| U_\e(y) - w_\e(y) \right| &= \left| g(\e x_0, y) - w_\e(y) \right| \\
&\le \left| g(\e x_0, y) - w_\e(y^\prime) \right| + \left| w_\e(y^\prime) - w_\e(y) \right| \\
&= \left| g(\e y, y) - g(\e x_0, y^\prime) \right| + \left| w_\e(y^\prime) - w_\e(y) \right| \\
\end{aligned} \end{equation}
where $y^\prime$ be the orthogonal projection of $y$ to the hyper-plane $\partial H \left(-\nu, y_{0,\e} \right)$.

Since $|y^\prime -y| \le  C_1 \e^{2p-1}$ and $ |(\e y, y) - (\e x_0, y^\prime) | \le C_2 \e^{2p-1}$, we get the result because of the regularity of $g$ and Lemma \eqref{cor-reg}. For $y \in \partial H(-\nu, y_{0,\e})$, the conclusion comes from the similar argument. 
\end{proof}

\begin{lemma} \label{cor-approx2}
Let $p$ and $q$ be constants satisfying $\frac{1}{2} < p < q < 1$. Let $Q^\prime$ be a cube on $\partial \R^{n-1}$ centered at the origin with side length $\e^{p-1}$ and let $Q = Q^\prime \times (0,\e^{p-1})$. Suppose that $h_\pm$ are the solutions to the following equation,
\begin{equation} \label{eq-cor-approx2} \begin{cases}
\cM^\pm(D^2 h_\pm) = 0 &\text{ in } Q \\
h_\pm =0  &\text{ on } Q^\prime \\
h_\pm = \pm1 &\text{ on } \partial Q \setminus Q^\prime.
\end{cases} \end{equation}

Then, 
\begin{equation}
\|h_\pm\|_{L^\infty(B_{\e^{q-1}}(0))} \le C \e^{q-p}
\end{equation}
where $C$ is a consatant depending only on the $n$, $\lambda$, and $\Lambda$. 
\end{lemma}
\begin{proof}
Consider the following equation
\begin{equation}
\widetilde h_+(x) = \left| \frac{x^\prime}{\e^{p-1}} \right|^2 + \displaystyle\frac{(n-2) \Lambda}{\lambda} \left( - \left( \frac{x_n}{\e^{p-1}} - 1 \right)^2 + 1 \right)
\end{equation}
where $\Lambda$ and $\lambda$ are elliptic constants in \eqref{con-F-unif}.
By the direct calculation, we can check that $\cM^+(D^2\widetilde h_+) \le 0$ and $\widetilde h_+ \ge 1$ on $\partial Q$. That implies 
$\widetilde h_+ \ge h_+$ in $Q$ from the comparison. In the similar way, we can show that  $-\widetilde h_+ \le h_-$. Hence we have the following,
\begin{equation}
\| v \|_{L^\infty(B_{\e^{q-1}}(0))} \le \| h \|_{L^\infty(B_{\e^{q-1}}(0))} \le C  \e^{q-p}
\end{equation}
where $C$ is a constant depending only on the $n$, $\lambda$, and $\Lambda$. 
\end{proof}

\begin{lemma} \label{cor-approx3}
Let $U_\e$ be the solution of Equation \eqref{eq-cor-approx} and let $w_\e$ be the solution of \eqref{eq-cor1}. Then, for  $\frac{1}{2} < p < q < 1$ satisfying $2p-1 \le q-p$, we have
\begin{equation}
\left\| U_\e(y) - w_\e(y) \right\|_{L^\infty( \e^{-1} D \,\cap\, H (-\nu, y_{0,\e} ) \,\cap\, B_{\e^{q-1}} ( y_{0,\e} ) )} \le C \e^{2p-1} 
\end{equation}
where $C$ depends only on $n$, $\lambda$, $\Lambda$ and the $\mathcal{C}^2$ norm of $\partial D$ and $\nu =-\nu(x_0)$ is the inward normal vector at $x_0$. 
\end{lemma}
\begin{proof}
We assume that $x_0 = 0$ and $\nu = e_n$ without any loss of generality and the result for general $x_0 \in \partial D$ and $\nu$ can be obtained by a translation and a rotation of the domain. Also, we ignore the term $\e^2 f(\e y,y)$ in the equation \eqref{eq-cor-approx} because the size of error is $o(\e)$. Let $Q$ be a cube same as in Lemma \ref{cor-approx2}.  Let $ v= U_\e-w_\e$. Note that from \cite{CC}, $v \in \mathcal{S}(\lambda,\Lambda)$ that means $v$ is a solution of a uniform elliptic operator with elliptic constant $\lambda$ and $\Lambda$. In addition, from Lemma \ref{cor-approx1} and from the fact that $|U| + |w| \le 2 |g|_\infty$, $|v| \le 2|g|_\infty$ on $\partial Q \setminus Q^\prime$ and $|v| \le C_1 \e^{2p-1}$ on $Q^\prime$ where $C_1$ is a constant which is the same in the proof of Lemma \ref{cor-approx1}. Note that the function 
\begin{equation}
v_\pm(x) = \pm C_1 \e^{2p-1} + 2 |g|_\infty h_\pm(x) 
\end{equation}
be a (viscosity) super- and sub-solution respectively where $h_\pm(x)$ be functions defined same in Lemma \ref{cor-approx2}. Moreover, because of the definition of $v^\pm$, $v_- \le v \le v_+$ holds on $\partial Q$ and hence the following holds for all $y \in Q$,
\begin{equation}
v_- \le v \le v_+
\end{equation}
and it implies the conclusion. 
\end{proof}

\section{The Continuity of $\overline g$} \label{sec-gconti}
We denote $w_\e(y;x_0,\nu)$ if it is a solution of the equation \eqref{eq-cor1} for a point $x_0 \in \partial D$ and a direction $\nu \in S^{n-1}$. Additionally, $w_\e(y;x_0)$ is defined by $w_\e(y;x_0,-\nu(x_0))$ where $\nu(x_0)$ is a outward normal vector of $D$ at a point $x_0$. We apply the same notation convention to the effective boundary data $\overline g(x_0,\nu)$ hence $\overline g(x_0) = \overline g(x_0,-\nu(x_0))$. 

The goal of this section is to show the continuity of $\overline g$. As we mentioned in the previous section, $\overline g$ is well defined when $\nu(x)$ is irrational. 
\begin{prop} \label{prop-gconti}
$\overline  g(x,\nu)$ defined the same as in the Definition \ref{def-cor-overg}, is continuous on $\partial D \times \{ \nu \in S^{n-1} : \nu \in \mathcal{IR}\}$. 
\end{prop}

For a given direction $\nu_0 \in S^{n-1}$, the hyper-plain $\partial H(\nu_0,0)$ is well defined and it can be represented as a union of disjoint cubes $Q_R(k^\prime,\nu_0)$, $k^\prime \in \mathbb{Z}^{n-1}$ with radius $R$. According to Lemma \ref{hyper-ap}, for each irrational $\nu_0 \in S^{n-1}$ and $\delta>0$, There exists a constant $R>0$ such that each cube $Q_R(k^\prime)$ has a point $\hat y(k^\prime,\nu)$ satisfying
\begin{equation}| g(x_0,y + \hat y(k^\prime,\nu)) -g(x_0,y) |  \le \delta
\end{equation}
for all $y \in \R^n$. The following lemma tells us that, for given any $\nu_0$, we can choose $R$ independently on the choice of $\nu$ in a small neighborhood of $\nu_0$. 
\begin{lemma} \label{gconti-ap}
Suppose that $\nu_0$ is irrational. Then for given any $\delta >0$, there exist a constant $R$, a neighborhood $B_\eta(\nu_0) \cap S^{n-1}$ of $\nu_0$, and a family of mutually disjoint cubes $\{ Q_R(k^\prime,\nu) : k^\prime \in \mathbb Z^{n-1}, \nu \in B_\eta(\nu_0) \cap S^{n-1} \}$, satisfying
\begin{enumerate}
\item
each cube $Q_R(k^\prime,\nu)$ has side length $R$, and 
\item
$\bigcup_{k^\prime \in \mathbb Z^{n-1}} \overline Q_R(k^\prime,\nu) = H(\nu,0)$,
\end{enumerate}
such that, for each given $\nu \in B_\eta(\nu_0) \cap S^{n-1}$ and $k^\prime \in \mathbb{Z}^{n-1}$, there is a point $\hat y(k^\prime,\nu) \in Q_R(k^\prime,\nu)$ satisfying $|\hat y(k^\prime,\nu) -\widetilde m |\le \delta$ for some $\widetilde m \in \mathbb{Z}^n$. 
\end{lemma}
\begin{proof}
For simplicity, assume that $|\nabla g| \le 1$ and $(\nu_0)_n >0$. Since we will choose $\eta$ small, we may assume that $\nu_n \ge \frac{(\nu_0)_n}{2}$ for all $\nu \in B_\eta(\nu_0) \cap S^{n-1}$ So, $\partial H(\nu,0)$ can be represented as a hyperplain
\begin{equation}
y_n = h(y^\prime,\nu) = -\displaystyle\frac{\nu_1}{\nu_n} y_1 - \cdots - \frac{\nu_{n-1}}{\nu_n} y_{n-1}.
\end{equation}

Let $Q^\prime_R(k^\prime) = k^\prime R + Q^\prime_R$ where  $k^\prime \in \mathbb{Z}^{n-1} \times \{0\}$ and $Q^\prime_R$ is a cube in $\R^{n-1} \times \{0\}$ centered at the origin with side length $R$. Then, for given any $m \in (\mathbb{Z} ^{n-1}\times \{0\}) \cap Q^\prime_R(k^\prime)$, we let 
\begin{equation}
t(m,\nu) = \text{ the fractional part of } h(m,\nu).
\end{equation} 
Since $\nu_0$ is irrational, from Lemma \ref{hyper-ud}, $t(m,\nu_0)$ is uniformly distributed on $\R / \mathbb{Z}$, and hence we can choose $R_0$ such that for any interval $I$ in $\R / \mathbb{Z}$ whose length is equal to $\delta/3$, there exists $m(k^\prime) \in Q^\prime(k^\prime) \cap \mathbb{Z}^{n-1} \times \{0\}$ such that $t(m,\nu_0) \in I$ for each $k^\prime$. 

Let $c(k^\prime) = h(k^\prime,\nu)- h(k^\prime,\nu_0)$ for given $\nu$. We choose $\eta > 0$ small that 
\begin{equation} \label{eq-gconti-ap}
| h(y,\nu) -h(y,\nu_0) - c(k^\prime) |\le \frac{\delta}{3}
\end{equation}
for all $y^\prime \in Q^\prime_{R_0}(k^\prime)$ and $\nu \in B_\eta(\nu_0) \cap S^{n-1}$.

Let $I = - c(k^\prime) + (\delta/3, 2\delta/3) \in \mathbb{R}/\mathbb{Z}$ and consider the following representation 
\begin{equation}
h (m,\nu) =\left(  h(m,\nu) - h(m,\nu_0) - c(k^\prime) \right) +\left( (h(m,\nu_0) + c(k^\prime) \right).
\end{equation}

From \eqref{eq-gconti-ap}, $\left( h(m,\nu) - h(m,\nu_0) - c(k^\prime) \right)$ is less than equal to $\delta/3$ and, from the choice of $m$, the fractional part of $\left( (h(m,\nu_0) + c(k^\prime) \right)$ is between $\delta/3$ and $2\delta/3$. So, the fractional part $t(m,\nu)$ of $h(m,\nu)$ is less than equal to $\delta < 1$. 

Now choose $\hat y(k^\prime, \nu) = (m,h(m,\nu))$ then we are done.
\end{proof}

\begin{lemma} \label{gconti-1}
Suppose that $w_\e(y;x_0,\nu)$ are the solutions of the equation \eqref{eq-cor1}. For any $\delta>0$, there is a constant $t_0 >0$ such that,  if $t\ge t_0$, then the estimate 
\begin{equation} 
|  w_\e(y^\prime + t \nu;x_0,\nu) - \overline g(x_0,\nu) | \le C\|g\|_{\mathcal C^2} \delta
\end{equation}
holds for all $y^\prime \in H(\nu,0)$ and for all irrational $\nu$ contained in $B_\eta(\nu_0) \cap S^{n-1}$.
Here, $\eta$ is chosen same in the Lemma \ref{gconti-ap}, $C$ is a constant depending only on $n$, $\lambda$ and $\Lambda$, and $t_0$ is a constant depending only on $n$, $\lambda$, $\Lambda$, the radius $\eta$, the direction $\nu_0$ and $\delta$ . 
\end{lemma}
Note that the constant $t_0$ does not depend the choice of $\nu$. 
\begin{proof}
From the maximum principle, we have that 
\begin{equation} \label{eq-gconti-1}
\min_{ y^\prime \in \partial H(\nu,y_{0,\e}) } w_\e( y^\prime + t_0 \nu ;x_0,\nu) \le | w_\e( y^\prime + t \nu ;x_0,\nu) |\le \max_{y^\prime \in \partial H(\nu,y_{0,\e})} w_\e( y^\prime + t_0 \nu ;x_0,\nu) 
\end{equation}
for all $t$ and $t_0$ satisfying $t \ge t_0$. 
Hence we have
\begin{equation} \label{eq-gconti-2}
\min_{ y^\prime \in \partial H(\nu,y_{0,\e}) } w_\e( y^\prime + t_0 \nu ;x_0,\nu) \le \overline g(x_0,\nu) \le \max_{ y^\prime \in \partial H(\nu,y_{0,\e}) } w_\e( y^\prime + t_0 \nu ;x_0,\nu).
\end{equation}
By combining \eqref{eq-gconti-1} and \eqref{eq-gconti-2}, we have
\begin{equation}
\max_{y^\prime \in \partial H(\nu,y_{0,\e})} |  w_\e(y^\prime + t\nu;x_0,\nu) - \overline g(x_0,\nu) | \le \osc_{y^\prime \in \partial H(\nu,y_{0,\e})} w_\e( y^\prime + t_0\nu;x_0,\nu)
\end{equation}
if $t \ge t_0$. 

Choose $\eta$ same in the Lemma \ref{gconti-ap}. Then, from above and from the inequality \eqref{eq-hyper-cor} in the proof of Lemma \ref{hyper-osc}, the inequality
\begin{equation}
|  w_\e(y^\prime + t \nu) - \overline g(x_0) | \le 2 \left( \displaystyle\frac{3R}{t}\right)^{\log_2 \frac{1}{\gamma}} |w|_\infty + C \|g\|_{\mathcal{C}^2} \delta
\end{equation} 
holds for all $\nu \in B_\eta(\nu_0) \cap S^{n-1}$ where $\gamma$ and $C$ are constants same in the inequality \eqref{eq-hyper-cor} and $R$ is a constant same in the Lemma \ref{gconti-ap}.By choosing $t_0 =3R \delta^{1/\log_2\gamma}$, we get the conclusion from the fact that $|w|_\infty \le  \|g\|_{\mathcal{C}^2} $. 
\end{proof}

\begin{lemma} \label{gconti-2}
For a periodic function $g$, a direction $\nu_0$, and a constant $t>0$, there is a neighborhood $B_\eta(\nu_0)$ of $\nu_0 \in S^{n-1}$ such that if $\nu \in B_\eta(\nu_0) \cap S^{n-1}$, then 
\begin{equation}
| w_\e(t\nu;x_0,\nu) - w_\e(t\nu_0;x_0,\nu_0)| \le C \|g\|_{\mathcal C^2(\square)} \delta
\end{equation}
where $w(y;x_0,\nu)$ is the solution of the equation \eqref{eq-cor1} and $C$ is a constant depending only on $n$, $\lambda$ and $\Lambda$. 
\end{lemma}
\begin{proof}
Let $Q_M(0)$ be a large box of $\R^n$ centered at the origin with side length $M > 10t$. For given any $\delta$, we choose $\eta$ so small that $| w_\e(y;x_0,\nu) - w_\e(y;x_0,\nu_0)| \le \delta $ if $y \in Q_M(0) \cap \partial \Large( H(\nu,0) \cap H(\nu_0, 0) \Large)$. Note that as long as $M$ is chosen, we always find such a $\eta$ because of the regularity result in the Lemma \ref{hyper-reg1}. 

For a given $t>0$, there is a large $M$ such that 
\begin{equation}| w_\e(t\nu;x_0,\nu) - w_\e(t\nu;x_0,\nu_0) | \le 2 \delta
\end{equation}
from a similar argument as Lemma \ref{cor-approx2}.

By choosing $\eta \le \delta / t$, we have $|t\nu - t\nu_0| \le \delta$. From this and from the Lemma \ref{hyper-reg1}, we have
\begin{equation} \label{cor-approx22}
| w_\e(t\nu;x_0,\nu_0) -w_\e(t\nu_0;x_0,\nu_0) | \le C \|g\|_{\mathcal C^2 (\square)} \delta
\end{equation}
and then, we get the conclusion by combining above two inequalities.
\end{proof}

\begin{proof}[Proof of Proposition \ref{prop-gconti}]

For fixed direction $\nu$ and $x_0 \in \partial D$, let $\widetilde w(y) = \widetilde w(y;x_0;\nu)$ be a solution of the equation \eqref{eq-cor1} with $y_{0,\e} =0$. Because of the Lemma \ref{cor-lim}, we have
\begin{equation}
\lim_{t \rightarrow \infty} \widetilde w(t\nu;x_0,\nu) =\overline g(x_0,\nu)
\end{equation}
as long as $\nu$ is irrational. 

From the assumption \eqref{con-g}, it follows that 
\begin{equation}
\sup_y |g(x_1,y) - g(x_2,y)| \le \|g\|_{\mathcal C^1} |x_1-x_2|
\end{equation}
for any $x_1$ and $x_2 \in \partial D$. 
Hence, from the maximum principle in \cite{CLV}, we have
\begin{equation}
|\widetilde w(t\nu;x_1,\nu) - \widetilde w(t\nu;x_2,\nu)| \le \|g\|_{\mathcal C^1} |x_1-x_2|
\end{equation}
for all $t>0$. The conclusion comes by taking $t \rightarrow \infty$. 

Secondly, fix a point $x_0$ and $\delta$. We also fix a irrational $\nu_0$. Then, from Lemma \ref{gconti-1}, there exists a constant $\eta_1$ and $t_0$ such that 
\begin{equation} 
| w_\e(t \nu;x_0,\nu) - \overline g(x_0,\nu) | \le C\|g\|_{\mathcal C^2} \delta
\end{equation}
holds for all $t\ge t_0$ and for all irrational $\nu \in B_{\eta_1} (\nu_0) \cap S^{n-1}$ where $w_\e(y;x_0,\nu)$ is the solution of the equation \eqref{eq-cor1}.

Now by applying Lemma \ref{gconti-2} when $t=t_0$, there is $\eta_2>0$ such that if $\nu$ is irrational and satisfies $\nu \in B_{\eta_2}(\nu_0) \cap S^{n-1}$, then
\begin{equation}
| w(t_0\nu;x_0,\nu) - w(t_0\nu_0;x_0, \nu_0)| \le C \|g\|_{\mathcal C^2(\square)} \delta
\end{equation}
holds. 
Choose $\eta = \min\{\eta_1, \eta_2\}$ Then the following holds for all $\nu \in B_{\eta}(\nu_0) \cap S^{n-1}$,
\begin{equation} \begin{aligned}
| \overline g(x_0,\nu) - \overline g(x_0,\nu_0) | &\le | \overline g(x_0,\nu) - w(t \nu;x_0,\nu) | + | w(t \nu;x_0,\nu)  - w(t \nu_0;x_0,\nu_0) | \\
&\quad+ | w(t \nu_0;x_0,\nu_0) - \overline g(x_0,\nu_0) | \\
&\le C \|g\|_{\mathcal C^2 (\square)} \delta.
\end{aligned}\end{equation}
Here $C$ is a constant depending only on $n$, $\lambda$ and $\Lambda$. So it is continuous. 
\end{proof}

\section{The Effective Solution} \label{sec-overg}
We find a ¡°Effective Boundary Data¡± $\overline g$ in Section \ref{sec-cor}, and we showed that $\overline g$ is continuous on $\partial D \setminus \{x : \nu(x) \in \mathcal R \}$ in Section \ref{sec-gconti}. Nevertheless it could be discontinuous on $\partial D$. As long as we know, there is no concept of the viscosity solution when the boundary data is not continuous. That is why we define the generalized concept of the solution of the equation \eqref{eq-limit-2}. 

Let $D = \{ (x^\prime ,x_n) \in \R^n : |x^\prime|^2 + (x_n-1)^2 < 1 \}$, $g(y) = \cos y_n$ and $u_\e$ is the solution of the following equation,
\begin{equation} \label{eq-overg-ex} \begin{cases}
\Delta u_\e(x) = 0 &\text{ in } D,\\
u_\e = g(x/\e) &\text{ on } \partial D.
\end{cases} \end{equation}
By the calculation, we obtain
\begin{displaymath}
\overline g(x) = \left\{ \begin{array}{ll}
 1, &\text{ if } x =(0,0),\\
 \text{undefined,} &\text{ if } x=(0,2),\\
 0, &\text{ otherwise} 
 \end{array} \right.
\end{displaymath}

From the result in \cite{LS}, $u_\e$ converges to $0$ uniformly on every compact set. However, because of the Lemma \ref{cor-approx3}, we have 
\begin{equation} \label{eq-overg-ex}
| u_\e(0,\e^q) -w_\e(0,\e^{q-1};0,(0,1)) | = |  u_\e(0,\e^q) - 1 | \le C \e^{2p-1}
\end{equation}
where $p, q$ are constants satisfies the condition in Lemma \ref{cor-approx3}. It implies that $u_\e$ is very close to $1$ near the origin. This example tells us that the value of $u_\e$ in a small neighborhood of a point $0 \in \partial D$ does not make any serious oscillation at the interior.

\begin{definition} \label{def-overg-finite}
Let $x_1, x_2, \cdots, x_k $ be finite points on $\partial D$ and $g_m$, $m=1,2,\cdots$ be uniformly bounded and continuous functions on $\partial D$ supported only on $\partial D \cap \left( \bigcup_{i1,2,\cdots,k} B_{r_m} (x_i) \right)$ where $r_m$ is a sequence of constants converging to $0$. Let $b^\pm_m$ be solutions of the equation 
\begin{equation} \label{eq-overg-finite} \begin{cases}
\cM^\pm(D^2 b^\pm_m(x),\lambda,\Lambda) = 0 &\text{ in } D,\\
b^\pm_m(x) = g_m(x) &\text{ on } \partial D
\end{cases} \end{equation}
where $\cM^\pm$ is the Pucci operator with elliptic constants $\lambda$ and $\Lambda$. 
Then we say that a Dirichlet problem \eqref{eq-main} is in the stable  class for finite values of the boundary if $b^\pm_m$, the solution of the equation \eqref{eq-overg-finite} converge to zero for all compact subset $K$ of $D$. 
\end{definition}

We note that, by using \eqref{eq-overg-ex} and the Lemma \ref{gconti-dom}, it can be shown that the above condition is necessary to guarantee the uniqueness of the limit $u_0$ of $u_\e$. 

\begin{lemma} \label{overg-finite}
Suppose that $(n-1)\lambda > \Lambda$. Then, \eqref{eq-main} is stable for the finites value of the boundary. 
\end{lemma}
When the operator is Laplacian, $\Lambda / \lambda =1$ hence the condition $(n-1)\lambda > \Lambda$ satisfies, but if the operator is degenerate, then $(n-1)\lambda > \Lambda$ cannot satisfies. So, this condition implies that the operator is ¡°sufficiently uniformly elliptic¡±. 

\begin{proof}
Let us assume that $k=1$, $x_1=0$ and $|g_m|_\infty \le 1$ for all $m \in \mathbb{Z}$. From $(n-1)\lambda > \Lambda$, there is a homogeneous solution $h(x) = 1/|x|^\alpha$, $\alpha = (n-1)( \lambda / \Lambda) -1 >0$, of $\cM^+$. Since $g \le r_m^\alpha h(x) = r_m^\alpha / |x|^\alpha$ and $ r_m^\alpha / |x|^\alpha$ is a solution of $\cM^+$, $b^+_m \le  r_m^\alpha / |x|^\alpha$ because of the comparison. Hence we have
\begin{equation}
\sup_K b^+_m  \le \displaystyle\frac{r_m^\alpha}{\dist(x_1, K)^\alpha}.
\end{equation}
Similarly, we can show that 
\begin{equation}
\inf_K b^-_m  \ge -\displaystyle\frac{r_m^\alpha}{\dist(x_1, K)^\alpha}
\end{equation}
by using a homogeneous solution $-1/|x|^\alpha$ of $\cM^-$.
So,$ \| u^\pm_m \|_{L^\infty(K)}$ converges to 0 as $r_m$ converges to 0.  

Now let us consider the case when $k >1$. Since we consider the limit behavior as $r_m \rightarrow 0$, we may assume that $B_{r_m}(x_i)$ are mutually disjoint. Let $h^i_m(x) = r_m^\alpha / |x-x_i|^\alpha$ and $h_m(x) = \sum_i h^i_m(x)$. Then we have
\begin{equation}
\cM^+(D^2 h_m) \le \sum_i \cM^+(D^2 h^i_m) = 0.
\end{equation}
Since it is obvious that $b^+_m(x) \le h_m(x)$ on $\partial D$, we have $b_m^+(x) \le h_m(x)$ in $D$. So, we have
\begin{equation}
\sup_K b^+_m(x) \le \displaystyle\frac{r_m^\alpha}{\sum_{i=1}^k \dist(x_i, K)^\alpha}.
\end{equation}
Similarly, we obtain 
\begin{equation}
\inf_K b^-_m(x) \ge -\displaystyle\frac{r_m^\alpha}{\sum_{i=1}^k \dist(x_i, K)^\alpha}.
\end{equation}

By combining above two inequalities and the fact that $b_m^-(x) \le b_m^+(x)$, we get the conclusion. 
\end{proof}

Through this section, we assume that \eqref{eq-main} is in the stable class for finite values of the boundary. Let us consider the following equation 
\begin{equation} \label{eq-overg-main} \begin{cases}
\overline F(D^2 \overline u(x)) = \overline f(x) &\text{ in } D,\\
\overline u(x) = \overline g(x) &\text{ on } \partial D
\end{cases} \end{equation}
where $\overline g(x)=\overline g(x;-\nu(x))$ is the function defined as in the Definition \ref{def-cor-overg} on every irrational point $x \in \partial D$.

It is same with the equation \eqref{eq-limit-2}. We just rewrite it. It is well known that the function $\overline F$ and $\overline f$ is well defined. It is also well-known that $\overline F$ is uniformly elliptic with the same elliptic constant $\lambda$ and $\Lambda$. So, as long as $\overline g$ is continuous, the above problem is well-posed. 

For simplicity, we assume the the function $g$ in the equation \eqref{eq-main} is independent on $x$ variable. We note that it does not make any serious change of proofs. For given any rational direction $\nu$, there is a integer point $m =(m_1,m_2,\cdots,m_n)$ which is parallel to $\nu$ such that there are no common integer factors of $m_1, \cdots m_n$ greater than $1$. Let 
\begin{equation}
\mathcal D_\delta = \{\nu \in \mathcal R : m_i > 1/\delta \text{ for some } i=1,2,\cdots, n \}. 
\end{equation}
We note that there are only finite rational directions which are not contained in $\mathcal D_\delta$ for given $\delta>0$. 

If $\nu_0 \in \mathcal D_\delta$, then $\nu_0$ has similar properties that the irrational direction has for given $\delta>0$. So, we have the following which is very similar to the Lemma \ref{gconti-ap}. 
\begin{lemma} \label{overg-ud}
Suppose that $\nu_0 \in \mathcal D_\delta$. Then there exist a neighborhood $B_\eta(\nu_0)$ and a constant $R >0$ which does not depend on $\nu$ but $\delta$ and $\nu_0$ such that for any given $\nu \in B_\eta(\nu_0) \cap S^{n-1}$ and $k^\prime \in \mathbb Z^{n-1}$, there is a point $\hat y(k^\prime, \nu) \in Q_R(k^\prime,\nu)$ satisfying $|\hat y(k^\prime, \nu) - \widetilde m| \le \delta$ for some $\widetilde m \in \mathbb Z^n$ where $Q_R(k^\prime,\nu)$ is defined same as in Lemma \ref{gconti-ap}.
\end{lemma}
\begin{proof}
We will prove the case when $m_1 > 1/\delta$ and the other case can be proved by the similar way. Since $m_1\neq \pm 1$, we always find a nonzero element among $m_2, \cdots, m_n$, so we assume also that $m_n\neq0$. We may represent the set  $\{ y \in \R^n | \nu_0 \cdot y =0\}$  by 
\begin{equation} \label{eq-overg-1}
y_n = h(y^\prime) = -\frac{m_1}{m_n} y_1 - \cdots -\frac{m_{n-1}}{m_n} y_{n-1}.
\end{equation}

Choose $R$ to satisfy $2 m_1 \le R$ and choose a sequeqnce of finite integer points $z_i$, $i=0,1,\cdots, m_1$ defined by $z_i = z_0 + i e_1$ for some $z_0 \in \R^{n-1}\times \{0\}$. Because of the choice of $R$, we may choose all the $z_i$ in the cube $Q^\prime_R(k^\prime,\nu_0)$ for fixed $k^\prime \in \mathbb Z^{n-1} \times \{0\}$. Let $t_i$ be the fractional part of $h(z_i)$. Since $\nu_0$ is rational, $t_i$ are distributed uniformy in $[0,1)$. So we can find a point $z_i$ such that $t_i \le \delta$ because of the relation $m_1 > 1/\delta$. Now let $\hat y(\nu_0,k) = (z_i, h(z(i))$ and $\widetilde m = (z_i,h(z_i) - t(z_i))$. Then we get the conclusion for fixed $\nu_0$.

The remaining of the proof is quite similar to the proof of Lemma \ref{gconti-ap} so we omit it. 
\end{proof}

Let $u_\e$ be the solution of the equation \eqref{eq-main}. By the regularity theory of the viscosity solution, \cite{CC}, we have
\begin{equation}
\|u_\e\|_{\mathcal C^{1,\alpha}(K)} \le C \left( \|f\|_{L^\infty(D)} + \|g\|_{L^\infty(\partial D)} \right)
\end{equation}
for some compact $K \subset D$. So, we always find a limit $u_0$ such that $u_{\e_j}$ converges to $u_0$ uniformly on $K$. 

\begin{theorem} \label{thm-overg-main}
Let $g(y)$ be function which is same in the equation \eqref{eq-main}, $\overline g(x)=\overline g(x;-\nu(x))$ be function defined on $\{x \in \partial D ; -\nu(x) \in \mathcal{IR}\}$, and $g^+(x) \in \mathcal C^0(\partial D)$ be a function satisfying $\overline g(x) \le g^+(x)$ for all $x$ in $\{ x \in \partial D :  -\nu(x) \in \mathcal{IR}\}$.  Suppose that $u^+$ is the solution of the equation \eqref{eq-overg-main}, when the boundary data is $g^+(x)$. Then 
\begin{equation}
u_0 \le u^+ \text{ on } K
\end{equation}
for all compact subset $K$ of $D$ where $u_0$ is any limit of subsequence of $u_\e$. 
\end{theorem}

Note that the other side of inequality also holds, so we get the following. 
\begin{cor} \label{cor-overg-main}
If there is a continuous function $g_0$ such that $g_0(x) = \overline g(x)$ for all irrational points $x \in \partial D$, then the solution $u_\e$ of \eqref{eq-main} converges to $u_0$, the solution of the equation \eqref{eq-overg-main} when the boundary data is $g_0$, uniformly on every compact set $K \subset D$. 
\end{cor}

For given any $x_0 \in \partial D$ satisfying $\nu_0 = -\nu(x_0) \in \mathcal D_\delta$, we let $B_\eta(\nu_0)$ and $R$ is defined same in the Lemma \ref{overg-ud}. We note that the size of the cube $R$ determines the size of oscillation. See Lemma \ref{gconti-1} or, more precisely, for all $x \in \partial D$ satisfying $\nu = -\nu(x) \in B_\eta(\nu_0)$, 
\begin{equation}
W_t(x) = \displaystyle\osc_{y \in t\nu+ \partial H(\nu,y_{0,\e})} w_\e(y;x) \le \left( \displaystyle\frac{3R}{t}\right)^{\log_2 \frac{1}{\gamma}} |g|_\infty + C \|g\|_{\mathcal{C}^2} \delta
\end{equation}
where $\gamma \in (0,1)$ be a uniform constant and $w(y;x)$ is the solution of the equation \eqref{eq-cor1} with $\nu = -\nu(x)$.

From the maximum principle, $W_t(x)$ is monotone decreasing as $t \rightarrow \infty$ and we have
\begin{equation}
w_\e(y_{0,\e}+t\nu;x,\nu) \le \overline g^*(x,\nu) + W_t(x) \le \overline g^*(x) + \left( \displaystyle\frac{3R}{t}\right)^{\log_2 \frac{1}{\gamma}} |g|_\infty + C \|g\|_{\mathcal{C}^2} \delta.
\end{equation}

From Lemma \ref{overg-ud}, we choose $R$ uniformly on the direction $\nu$ in a small neighborhood $B_\eta(\nu_0)$. So, if we choose $r>0$ small enough, we can conclude the following.
\begin{lemma} \label{overg-1}
Let $x$ be a point on $\partial D$ satisfying $\nu_0 = -\nu(x_0) \in \mathcal D_\delta$ or $\nu_0$ is irrational. Then there exist constants $r >0$ and $t_0>0$ which depend only on $n$, $\lambda$, $\Lambda$, $\nu_0$ and $\delta$ such that the following inequality,
\begin{equation}
w_\e(y_{0,\e}+t\nu;x) \le \overline g^*(x) + C_1 \|g\|_{\mathcal{C}^2} \delta,
\end{equation}
holds for all $x \in B_r(x_0) \cap \partial D$ where $\nu = -\nu(x)$ and $C_1$ is  constant dependsing only on $n$, $\lambda$ and $\Lambda$. 
\end{lemma}

The following two lemmas are about the behaviors of $\overline g$, $\overline g^*$ and $\overline g_*$. 
\begin{lemma} \label{overg-dconti}
For each $x_0 \in \partial D$ satisfying $-\nu(x_0) \in \mathcal D_\delta$, there exists a neighborhood $B_r(x_0)$ such that if $x_1, x_2 \in B_r(x_0)$ are irrational, then 
\begin{equation}
| \overline g(x_1) - \overline g(x_2) | \le C \| g\|_{\mathcal C^2(\square)} \delta
\end{equation}
where the constant $C$ only depends on $n$, $\lambda$ and $\Lambda$. 
\end{lemma}
The proof of the above lemma is quite similar to the proof of Proposition \ref{prop-gconti}. See the last sentence of the proof of Proposition \ref{prop-gconti}. 

\begin{lemma} \label{overg-2}
For any rational point $x_0 \in \partial D$ satisfying $\nu(x_0) \in \mathcal D_\delta$, 
\begin{equation}
\overline g^*(x_0) \le \limsup_{x \rightarrow x_0} \overline g(x) + C_2 \delta,
\end{equation}
for some constant $C_2$ depending only on $n$, $\lambda$, and $\Lambda$.
\end{lemma}

Since $\overline g$ is defined only on the irrational point, $\limsup_{x \rightarrow x_0} \overline g(x)$ means 
\begin{equation}
\lim_{r\rightarrow0} ~\sup_{x \in A_r} ~\overline g(x)
\end{equation}
where $A_r = \{ x \in B_r(x_0) \cap \partial D :  \nu(x) \text{ is irrational} \}$. 

Its proof is also quite similar to the proof of Proposition \ref{prop-gconti}. We only give an idea of proof. 
\begin{proof} [Sketck of the Proof of Lemma \ref{overg-2}]
Because of the definition of $\overline g^*(x_0)$, we may choose a value $y_0$ such that $\widetilde w(y;x_0)$, the solution of the equation \eqref{eq-cor1} with conditions $y_{0,\e} = y_0$ and $\nu = -\nu(x_0)$, satisfies 
\begin{equation}
\lim_{t\rightarrow\infty} \widetilde w (y^\prime-t\nu(x_0)) = \overline g^*(x_0)
\end{equation} 
for all $y^\prime \in \partial H(-\nu(x_0),y_0)$. So, by comparing $\widetilde w(y;x_0)$ with correctors $\widetilde w(y;x)$, solution of the equation \eqref{eq-cor1} with the conditions $x_0=x$, $y_{0,\e} = y_0$ and $\nu = -\nu(x)$, we get the result. 
\end{proof}

Let $D_\e = \{ x \in D ~|~ \dist(x,\partial D) > \e^q \}$ where $q \in (1/2,1)$. Since $\partial D$ is in $\mathcal{C}^2$, $\partial D_\e$ is also in $\mathcal{C}^2$, and there exists a 1-1 correspondence between $\partial D$ and $\partial D_\e$ if $\e$ is small. Let $z_\e= x + \e^q(-\nu(x)) \in \partial D_\e$, and define $\widetilde g_\e(x) = u_\e(z_\e)$ where $u_\e$ is the solution of the equation \eqref{eq-main}.
\begin{lemma} \label{gconti-dom}
Let $D_\e$, $z_\e$ and $\widetilde g_\e$ be same in the above. Then, for all $q \in (1/2,1)$, we have
\begin{equation}
\|\widetilde u_\e -u_\e \|_{L^\infty(D_\e)} \le C \e^{q\alpha}
\end{equation}
where $u_\e$ is the solution of the equation \eqref{eq-main}, $\widetilde u_\e$ is the solution of 
\begin{equation} \label{eq-gconti-dom}\begin{cases}
F(D^2 \widetilde u_\e, \frac{x}{\e}) = f(x,\frac{x}{\e}) &\text{ in } D \\
\widetilde u_\e(x) = \widetilde g_\e(x) &\text{ on } \partial D,
\end{cases} \end{equation}
$\alpha = \displaystyle\frac{\Lambda}{\lambda}(n-1) -1$, and $C$ is a constant depending only on $n$, $\lambda$, $\Lambda$ and $D$. 
\end{lemma}
\begin{proof}
Since $\partial D$ is in $\mathcal{C}^2$, we can find a radius $r_0>0$ and for every $x_0 \in \partial D$, there exists a exterior ball $B$ with radius $r_0$ and $\overline D \cap \overline B = \{x_0\}$. Assume that $c_0 = x_0 + r_0 \nu$ is a center of $B$. Because of the choice of $\alpha$, we have
\begin{equation}
h(x) = \displaystyle\frac{1}{r_0^\alpha} - \frac{1}{|x-c_0|^\alpha}
\end{equation}
is a super-solution of $\cM^+$ in $D$ and $h(x) \ge 0$ on $\partial D$. 

So, we can deduce that 
\begin{equation} \label{eq-gconti-dom-1}
\widetilde u_\e(z_\e) \le \widetilde u_\e(x) + \left( \displaystyle\frac{1}{r_0^\alpha} - \frac{1}{(r_0 + \e^q)^\alpha }\right) = u_\e(z_\e) + \left( \displaystyle\frac{1}{r_0^\alpha} - \frac{1}{(r_0 + \e^q)^\alpha }\right) 
\end{equation}
for all $x \in \partial D$ (or for all $z_\e \in \partial D_\e$).

From \eqref{eq-gconti-dom-1} and from the fact that $\widetilde u_\e + h(x)$ is a super-solution of the equation \eqref{eq-gconti-dom}, we have 
\begin{equation}
\widetilde u_\e(x) \le  u_\e(x) + \left( \displaystyle\frac{1}{r_0^\alpha} - \frac{1}{(r_0 + \e^q)^\alpha }\right)
\end{equation}
in $D_\e$ because of the comparison.

Similarly, we could obtain 
\begin{equation}
\widetilde u_\e(x) \ge u_\e(x) - \left( \displaystyle\frac{1}{r_0^\alpha} - \frac{1}{(r_0 + \e^q)^\alpha }\right)
\end{equation}
in $D_\e$. 
Hence the conclusion comes by combining above two inequalities.
\end{proof}

\begin{proof}[Proof of the theorem \ref{thm-overg-main}] \item
We first assume that $g^+(x) \ge \overline g(x) + \sigma$ for all irrational point $x\in \partial D$ and for a positive constant $\sigma$. If it is true, then we easily deduce the theorem by taking $\sigma \rightarrow 0$. 

For the convenience, we assume that $\|g\|_{\mathcal C^2(\square)} + \|g^+\|_{L^\infty(\partial D)} \le 1$. Chosse a $\delta$ to satisfy $(C_1+C_2) \delta \le \sigma/8$ where $C_1$ and $C_2$ are constants in Lemma \ref{overg-1} and Lemma \ref{overg-2}. 

Let $\mathcal E_\delta$ be the set of all rational direction $\nu \in S^{n-1}$ which are not contained in $\mathcal D_\delta$. Since $\mathcal E_\delta$ is finite, we can find a constant $r$, a function $h(x)$, satisfying
\begin{equation} \begin{cases}
\cM^+(D^2 h) \le 0 &\text{ in } D,\\
h(x) = 1 &\text{ on } \cup_{x \in \{\nu(x) \in \mathcal E_\delta \}}, B_{2r_0}(x),\\
h(x) =0 &\text { on } \partial D \setminus \left( \cup_{x \in \{ \nu(x) \in \mathcal E_\delta \} }, B_{4r_0}(x) \right),\\
0 \le h \le \sigma/8 \text{ in } K
\end{cases} \end{equation}
for given any compact subset $K$ of $D$ because of the assumption that \eqref{eq-main} is int the stable class for the finite values of the boundary. 

For each rational point $x_0$ in $\partial D \setminus \left( \cup_{ x \in \{ \nu(x) \in \mathcal E_\delta \} }, B_{r_0}(x) \right)$, we can find a neighborhood $B_r(x_0)$ such that there exists $\e_0$ such that if $\e \le \e_0$ then 
\begin{equation}
w(y_{\e}+\e^{q-1} \nu ; x) \le \overline g^*(x) + C_1\|g\|_{\mathcal C^2(\square) } \delta
\end{equation}
holds for all $x \in B_r(x_0)$ because of the Lemma \ref{overg-1} where $\nu = -\nu(x)$ and $y_\e = x/\e$. Note that only the radius $r$ and $\e_0$ could be affected by the choice of the point $x$ and $\delta$.  Additionally, by applying the Lemma \ref{overg-2} in the above equation, we have
\begin{equation} \label{overg-3}
w(y_{\e}+\e^{q-1} \nu ; x) \le \limsup_{x \rightarrow x_0} \overline g(x) + (C_1 + C_2) \delta
\end{equation}
for all $x \in B_r(x_0)$.

We note that the statement in the above is also true when $x_0$ is a irrational point, because the result in Lemma \ref{overg-1} holds for given any $\delta > 0$ and 
\begin{equation} \label{overg-31}
\overline g^*(x_0) = \overline g(x_0) = \limsup_{x \rightarrow x_0} \overline g(x)
\end{equation}
if $\nu(x_0)$ is irrational. The latter inequality in \eqref{overg-31} holds because of the Proposition \ref{prop-gconti}.

Now, since $\partial D \setminus \left( \cup_{x \in \{\nu(x) \in \mathcal E_\delta \}}, B_{r_0}(x) \right)$ is compact, we extract a finite covering among $\cup_{x_i} B_{r_i}(x_i)$ so we can choose $\e_0$ such that \eqref{overg-3} holds for every x  in $\partial D \setminus \left( \cup_{x \in \{ \nu(x) \in \mathcal E_\delta \}}, B_{r}(x) \right)$ if $\e \le \e_0$. 

Applying Lemma \ref{cor-approx3} to \eqref{overg-3}, we have
\begin{equation}
u_\e(x_0 + \e^{q}\nu) \le C \e^{p+q-1} +  \limsup_{x \rightarrow x_0} \overline g(x) + (C_1 + C_2) \delta
\end{equation}
where $1/2<p<q<1$ are constant satisfying the condition in Lemma \ref{cor-approx3}. 
Hence, if $\e$ is small enough, then 
\begin{equation} \label{overg-4}
u_\e(x + \e^{q}\nu) \le \limsup_{x \rightarrow x_0} \overline g(x) + \sigma/4  \le g^+(x) -3\sigma/4 \le u^+(x + \e^q\nu) - \sigma/2 
\end{equation} 
for all  $x \in \partial D \setminus \left( \cup_{x \in \{\nu(x) \in \mathcal E_\delta \}}, B_{r_0}(x) \right)$. Note that the last inequality of \eqref{overg-4} comes from the regularity of $u^+$. 

Let $\widetilde u_\e$ be same in the Lemma \ref{gconti-dom} for given $u_\e$. From \eqref{overg-4} and from the definition of $h$, we have
\begin{equation} \begin{aligned}
\widetilde u_\e(x)&\le u^+(x + \e^q\nu) + h(x) - \sigma/2  \\
&\le g^+(x) + h(x) +  \left( \displaystyle\frac{1}{r_0^\alpha} - \frac{1}{(r_0 + \e^q)^\alpha }\right) -\sigma/2 \\
&\le g^+(x) + h(x) -\sigma /4
\end{aligned} \end{equation}
for all $x \in \partial D$ if $\e$ is small. Hence, from the comparison,
\begin{equation}
\widetilde u_\e(x) \le u^+_\e(x) + h(x) -\sigma/4
\end{equation} 
holds in $D$ where  $u^+_\e$ is the solution of the equation
\begin{equation} 
\begin{cases}
F\left(D^2u^+_\e,\frac{x}{\e}  \right)= f\left( x,\frac{x}{\e} \right) &\text{ in } D \\
u^+_\e=g^+\left(\frac{x}{\e} \right)&\text{ on } \partial  D.
\end{cases}
\end{equation}
Note that it is well known that $u^+_\e$ converges to $u^+$ uniformly on $D$. See \cite{E}. 

From Lemma \ref{gconti-dom}, we have
\begin{equation}
u_\e(x) \le u^+_\e(x) + C\e^{q\alpha} + h(x) -\sigma/4 
\end{equation}
in $D_\e$. Hence, it also hold in $K$ since $K$ is contained in $D_\e$ if $\e$ is small enough. 

Now, take $\e \rightarrow 0$ on both side, we have 
\begin{equation}
u_0(x) \le u^+(x)
\end{equation}
on $K$. 
\end{proof}

\section{Proof of Theorem \ref{thm-main-2}} \label{sec-thm1}
As we discussed in Section \ref{sec-overg}, as long as the effective boundary data $\overline g$ has a continuous extension, we conclude that the solution $u_\e$ converges to $\overline u$ uniformly on every compact subset $K$ of $D$ where $\overline u$ is the unique solution of the equation \eqref{eq-limit-2} in the classical sense. In a Laplacian case, we easily prove that 
\begin{equation} \label{eq-thm1-av}
\overline g(x) = \langle g \rangle = \int_{[0,1]^n} g(y) dy
\end{equation}
so it has continuous extension. Moreover, in Section \ref{sec-cor}, we give an sufficient condition to have \eqref{eq-thm1-av}. However, it is hard to prove \eqref{eq-thm1-av} in general situation. So it is still open. In this section, we focus on the existence and uniqueness of $\overline u$ even if $\overline g$ does not have the continuous extension as long as the condition \ref{def-overg-finite} holds. 

\begin{definition} \label{def-thm1-dconti}
Suppose that $g(x)$ is a function defined on $\partial D$ except countably many points. Suppose also that $A$ is a compact subset of $\partial D$. We say $g$ is $\delta$-continuous for given $\delta$ on $A$, if for given any $x_0 \in A$, there exist a neighborhood $B_r(x_0)$ such that if $x_1,x_2 \in B_r(x_0)$ and $g(x_1)$, $g(x_2)$ are well defined, then 
\begin{equation}
| g(x_1) - g(x_2) | \le \delta. 
\end{equation}
\end{definition}
By using Lemma \ref{overg-dconti}, we easily prove that $\overline g$ is $\delta$-continuous on $\partial D \setminus \left( \cup_{ x \in \{ \nu(x) \in \mathcal E_\delta \} }, B_{r}(x) \right)$ for every $r>0$.

\begin{lemma} \label{thm1-approx-1}
Let $A$ be a compact subset of $\partial D$ and $g$ is $\delta$-continuous on $A$. Then, there are continuous functions $h^\pm$ such that $h^-(x) \le g(x) \le h^+(x)$ on $A$ and $h^+ -h^- \le C \delta$ where $C$ is a uniform constant.
\end{lemma} 
\begin{proof}
We will prove the case when $A$ is a subset of $\R^{n-1}$. Since $\partial D$ is regular manifold, it can be easily extend the case when $A \subset \partial D$. 

Since $g$ is $\delta$-continuous on $A$, we can find $r = r(x)$ such that if $x_1, x_2 \in B_{2r}(x_0)$, then $| g(x_1) - g(x_2)| \le \delta$. Moreover, since $A$ is compact, there is a finite covering $\cup B_{r_i} (x_i)$, $i=1,2,\cdots,m$, of $K$. Let $r= \min \{r_1,r_2, \cdots, r_m\}$. 

Let $\phi$ be a standard mollifier function whose support is contained in $B_{r}(0)$. Let $h^+(x) = g * \phi + \delta$. Since $\cup B_{r_i} (x_i)$ is a covering of $K$, we can find  a ball $B_{r_i} (x_i)$ such that $x \in B_{r_i}$ for any $x \in A$. From the choice of $r$, $B_r(x) \subset B_{2r_i} (x_i)$ and hence if $y \in B_r(0)$ then 
\begin{equation}
|g(x) - g(x-y)| \le \delta.
\end{equation}

So, we have
\begin{equation}
g(x) \le h^+(x) = \int_{B_\eta(0)} g(x-y) \phi(y) dy + \delta \le g(x) + 2 \delta
\end{equation}
if $g(x)$ is well defined at $x \in A$. In this way, we can define $h^-$ satisfying 
\begin{equation}
g(x) - 2 \delta \le h^-(x)  \le g(x) .
\end{equation}
\end{proof}

\begin{lemma} \label{thm1-approx-2}
Let $K$ be a compact subset in $D$. Also, let $\overline g(x)$ be the function in the Definition \ref{def-cor-overg}. Then, for any given $\delta>0$, there are functions $h^\pm(x) \in \mathcal C (\partial D)$, a constant $r>0$, and finite subset $\{z_1,z_2,\cdots,z_m\}$ of $\partial D$ such that
\begin{enumerate}
\item
$h^-(x) \le \overline g(x) \le h^-(x)$ and $| h^\pm(x) | \le 3 \|g\|_{L^\infty(\square)}$ on $\partial D$,
\item
$h^+(x) - h^-(x) \le 2 \delta$ for every $x \in \partial D \setminus (\bigcup_{i} B_r(z_i))$,
\end{enumerate}
Moreover, if $u^\pm(x)$ are viscosity solutions of the equation \eqref{eq-limit-2} when the boundary data are $h^\pm(x)$ respectively, then we have the following estimate,
\begin{equation}
0 \le \| u^+(x) - u^-(x) \|_{L^\infty(K)} \le C \delta
\end{equation}
where $C$ is a uniform constant. 
\end{lemma}
\begin{proof}
Without any loss of generality, we may assume that $\| g \|_{\mathcal C^2(\square)} =1$. As we discussed before, there are finite points $z_1, z_2, \cdots, z_m$ such that $\nu(z_i)$ are not in $\mathcal D_\delta$. Let $b(x)$ be a positive function having value 1 on $\partial D \cap \big(\cup_i B_r(z_i)\big)$ and supported in $\partial D \cap \big(\cup_i B_{2r}(z_i)\big)$. Let $v$ be the solution of the following equation 
\begin{equation} \begin{cases}
\cM^+(D^2 v) = 0 &\text{ in } D,\\
v(x) =b(x) &\text{ on } \partial D. 
\end{cases} \end{equation}

According to out assumption, the solution of the equation $v$ satisfies
\begin{equation}
0 \le v \le \delta
\end{equation}
on $K$ if we choose $r>0$ small enough.

Since $\overline g$ is $\delta$-continuous on $\partial D \setminus \big(\cup_i B_r(z_i)\big)$, there are continuous functions $\widetilde h^\pm(x)$ that has the same property in the Lemma \ref{thm1-approx-1}. By extending properly, we may assume that $\widetilde h^\pm \in \mathcal C^0(\partial D)$ and the difference between $\widetilde h^+$ and $\widetilde h^-$ is less than $C \delta$.

Extend $\widetilde h^\pm(x)$ to the interior of $D$ by using the equation $\overline F(D^2 \widetilde h) = \overline f(x)$. Let $h^\pm = \widetilde h^\pm(x) \pm v$. We easily check that $h^\pm$ satisfies the property (1) and (2) in the statement. Since $h^+$ is a super solution of the equation \eqref{eq-limit-2}, we have 
\begin{equation}
h^-(x) \le u^-(x) \le u^+(x) \le h^+(x) \text{ in } D. 
\end{equation}
from the comparison. So, we have 
\begin{equation} \begin{aligned}
\le \| u^+(x) - u^-(x) \|_{L^\infty(K)}  &\le \| h^+(x) - h^-(x) \|_{L^\infty(K)}   \\
&\le  \| \widetilde h^+(x) - \widetilde h^-(x) \|_{L^\infty(K)} +  2\| v\|_{L^\infty(K)} \\
&\le C \delta + 2 \delta \\
&\le C \delta.
\end{aligned} \end{equation}
\end{proof}

\begin{proof}[Proof of Theorem \ref{thm-main-2}] \item
Let $h^\pm_i(x)$ and $u^\pm_i(x)$ be the same in the Lemma \ref{thm1-approx-2} when $\delta = 1/i$, $i \in \mathbb N$. We let
\begin{equation} \begin{aligned}
H^+_i(x) &= \min \{ h^+_1, h^+_2, \cdots, h^+_i \},\\
H^-_i(x) &= \max \{ h^-_1, h^-_2, \cdots, h^-_i \},
\end{aligned} \end{equation}
and $\widetilde u^\pm_i(x)$ be the solutions of the equation \eqref{eq-limit-2} equipped the boundary data $H^\pm_i(x)$. 

Since $\widetilde u^+_i(x)$ is monotone, there is a function $\overline u^+$ which is defined in $K$. Similarly, we can define $\overline u^-$. From the fact that $u^-_i(x) \le \widetilde u^-_i(x) \le \widetilde u^+_i(x) \le u^+_i(x)$ and Lemma \ref{thm1-approx-2}, we have
\begin{equation} \label{eq-thm1-uniq}
0 \le u^+_i(x) - u^-_i(x) \le C /i
\end{equation}
for some uniform constant $C$. Hence we have $\overline u^-(x) = \overline u^+(x)$ on $K$.

Let $\overline u(x,K)$ be the function defined as $\overline u(x,K) = \overline u^+(x)$ on $K$. We easily check that if $K_1 \subset K_2$, then $\overline u(x,K_1) = \overline u(x,K_2)$ on $K_1$ because the barrier function when $K=K_2$ also be a barrier when $K=K_1$. So, $\overline u(x)$ is well defined by defining $\overline u(x) = \overline u(x,K)$ where $K$ is any compact set containing $x$. 

We claim that $\overline u$ is a solution of the equation \eqref{eq-limit-2} in the general sense. Suppose that $v$ is a solution of the equation \eqref{eq-limit-2} with $v \ge \overline g$ on $\partial D$. Then, since $\widetilde u^-_i(x) \le \overline g(x) \le v(x)$ on $\partial D$ for all $i$, we have $\widetilde u^+_i(x) \le v(x)$ in $D$. Hence we have
\begin{equation}
\overline u = \overline u^-\le v
\end{equation}
for all compact subset $K$ of $D$. Hence the above holds for all $x \in D$. 
Since the other inequality comes from the similar argument, $\overline u$ is the solution of the equation \eqref{eq-limit-2} in the general sense. Moreover, the uniqueness of the solution of the equation \eqref{eq-limit-2} immediately comes from the estimate \eqref{eq-thm1-uniq}

Now we are done if we show $u_0 = \overline u$. From the Theorem \ref{thm-overg-main}, we have
\begin{equation}
u^+_i(x) \le u_0(x) \le u^-_i(x)
\end{equation}
in some $K$. So, $\overline u(x) = u_0(x)$ on $K$. 
\end{proof}

As we told before, the IDDC is essential because if the domain does not satisfy the IDDC, we cannot expect the uniform limit. The following is a typical example for that.
\begin{example}
Suppose that $D =\{ (x_1,x_2) \in \R^2 : x_1^2 + (x_2-1)^2 <1, x_2 > 1\}$ and $g$ is given same in the Example \ref{ex-cor-overg}.  We denote $\partial_1 D = \partial D \cap \{x_2=1\}$ and $\partial_2 D = \partial D \setminus \partial_1 D$. As we told before, if we choose $\e = 1/2m$ for some integer $m$, then $g(x/\e) =1$ on $\partial_1 D$. By using the similar argument in the Lemma \ref{cor-approx2}, we can find a small $\eta>0$ which is independent on $\e$ such that $u_\e \ge 2/3$ in $B_\eta((0,1))$ where $u_\e$ is the solution of the equation \eqref{eq-main} when $D$ and $g$ are given by the same in the above. Similarly, if we choose $\e = 1/(2m+1)$, the $u_\e \le 1/3$ in $B_\eta((0,1))$. So, $u_\e$ cannot converges to some function locally uniformly on every compact subset of $D$. 
\end{example}


\begin{thebibliography}{99}
\bibitem[AL]{AL} Avellaneda, Marco; Lin, Fang-Hua,
	Compactness methods in the theory of homogenization. 
	Comm. Pure Appl. Math. 40 (1987), no. 6, 803¡©847.
\bibitem[BDLS]{BDLS}Barles, G.; Da Lio, F.; Lions, P.-L.; Souganidis, P. E. 
	Ergodic problems and periodic homogenization for fully nonlinear equations in half-space type domains with Neumann boundary conditions.  
	Indiana Univ. Math. J.  57  (2008),  no. 5, 2355???2375.
\bibitem[BM]{BM} G. Barles; E. Mironescu.
	On homogenization problems for fully nonlinear equations with oscillating dirichlet boundary conditions. 
	In preparation.
\bibitem[CC]{CC} X. Cabre, L. Caffarelli,
	Fully Nonlinear Elliptic Equation, 
	Vol. 43. American Mathematical Society. Colloquium Publication.(1983).
\bibitem[CIL]{CIL} Crandall, Michael G.; Ishii, Hitoshi; Lions, Pierre-Louis 
	User's guide to viscosity solutions of second order partial differential equations.
	Bull. Amer. Math. Soc. (N.S.)  {\bf 27}  (1992),  no. 1, 1--67		
\bibitem[CKL]{CKL}Choi, Sunhi; Kim, Inwon; Lee, Ki-ahm
	Homogenization of Fully Nonlinear Equations with Oscillating Neumann condition.
	in preparation. 
\bibitem[CLV]{CLV} Capuzzo-Dolcetta, I.; Leoni, F.; Vitolo, A.,
	The Alexandrov-Bakelman-Pucci weak maximum principle for fully nonlinear equations in unbounded domains.(English summary) 
	Comm. Partial Differential Equations 30 (2005), no. 10-12, 1863¡©1881. 
\bibitem[CSW]{CSW} Caffarelli, Luis A., Souganidis, Panagiotis E., Wang, L.,
	Homogenization of fully nonlinear, uniformly elliptic and parabolic partial differential equations in stationary ergodic media.
	Comm. Pure Appl. Math. 58 (2005), no. 3, 319¡©361.
\bibitem[E]{E} Evans, Lawrence C.
	Periodic homogenisation of certain fully nonlinear partial differential equations.
	Proc. Roy. Soc. Edinburgh Sect. A 120 (1992), no. 3-4, 245¡©265.
\bibitem[GM]{GM}Gerard-Varet, David ¢¥ ; Masmoudi, Nader,
	Homogenization and boundary layers.
	Acta Math. 209 (2012), no. 1, 133¡©178.1871-2509
\bibitem[GT]{GT} David Gilbarg and Neil S. Trudinger,
	Elliptic partial differential equations of second order, 
	Classics in Mathematic. Springer-Verla, Berlin, 2001.
\bibitem[JKO]{JKO} Jikov, V. V.; Kozlov, S. M.; Oleinik, O. A.,
	Homogenization of differential operators and integral functionals.
	Translated from the Russian by G. A. Yosifian, Springer-Verlag, Berlin, 1994, xii+570 pp. 
\bibitem[LS]{LS} Ki-ahm Lee, Henrik Shahgholian
	Homogenization of the boundary value for the Dirichlet Problem 
	in preparation.
\bibitem[LS1]{LS1} Lions, P.-L.; Souganidis, P. E. 
	Correctors for the homogenization of Hamilton-Jacobi equations in the stationary ergodic setting. 
	Comm. Pure Appl. Math. {\bf  56} (2003), no. 10, 1501???1524. 
\bibitem[LS2]{LS2} Lions, Pierre-Louis; Souganidis, Panagiotis E. 
	Homogenization of degenerate second-order PDE in periodic and almost periodic environments and applications.  
	Ann. Inst. H. Poincar¡îƒ Anal. Non Lin¡îƒaire  22  (2005),  no. 5, 667???677. 
\bibitem[LSY]{LSY} Ki-ahm Lee, Martin Str\"omqvist and Minha Yoo, 
	Highly oscillating thin obstacles(in prepation).
\bibitem[LT]{LT}  Lieberman, Gary M.; Trudinger, Neil S. 
	Nonlinear oblique boundary value problems for nonlinear elliptic equations.
	Trans. Amer. Math. Soc. {\bf  295 } (1986),  no. 2, 509--546.
\end{thebibliography}
\end{document}